\documentclass[12pt]{amsart}

\usepackage{subcaption}
\usepackage{fancyhdr,lipsum}
\usepackage{amssymb,amsmath,graphicx,color,textcomp, amsthm,bbm,bbold, enumerate,booktabs}
\usepackage{apptools}
\AtAppendix{\numberwithin{theorem}{section}}
\definecolor{Red}{cmyk}{0,1,1,0}

\definecolor{verde}{cmyk}{1,0,1,0}

\definecolor{loka}{cmyk}{.5,0,1,.5}
\definecolor{azul}{cmyk}{1,1,0,0}

\usepackage{hyperref}

\numberwithin{equation}{section}

\def\cal{\mathcal}

\newcommand{\Lbrace}{\left \lbrace}
\newcommand{\Rbrace}{\right \rbrace}

\newcommand{\diam}{\mathrm{diam}}

\def\Ed{{\mathbb{E}}}
\def\Pd{{\mathbb{P}}}

\newcommand{\dist}{\mathrm{dist}}
\newcommand{\N}{\mathbb{N}}
\newcommand{\R}{\mathbb{R}}

\renewcommand{\P}{\mathbb{P}}

\newcommand{\e}{\varepsilon}

\newcommand{\be}{\begin{equation}}
\newcommand{\ee}{\end{equation}}

\newcommand{\phitil}{\psi}

\newtheorem*{theorem*}{Theorem}
\newtheorem{theorem}{Theorem}
\newtheorem{proposition}[equation]{Proposition}
\newtheorem{definition}{Definition}
\newtheorem{lemma}{Lemma}[section]
\newtheorem{corollary}[equation]{Corollary}

\title{Clustering and Cliques in P.A random graphs with edge insertion}

\author{Caio Alves$^1$ }
\address{$^1$  Oak Ridge National Lab, Tennessee, USA
\newline
e-mail: {\itshape \texttt{demagalhaesc@ornl.gov}}}

\author{Rodrigo Ribeiro$^{2}$}
\address{$^2$ University of Denver, Colorado, USA
\newline
e-mail: {\itshape \texttt{rodrigo.ribeiro@du.edu}}}

\author{R{\'e}my Sanchis$^3$}
\address{$^3$Departamento de Matem{\'a}tica, Universidade Federal de Minas Gerais, Av. Ant\^onio
Carlos 6627 C.P. 702 CEP 30123-970 Belo Horizonte-MG, Brazil
\newline
e-mail: {\itshape \texttt{rsanchis@mat.ufmg.br}}}

\date{\today \\
    $^1$ Oak Ridge National Lab, USA\\
   	$^2$ Mathematics Departament, University of Denver, USA. \\
    $^3$ Departamento de Matem{\'a}tica, Universidade Federal de Minas Gerais, Brazil.
}

\begin{document}

\begin{abstract} In this paper, we investigate the global clustering coefficient (a.k.a transitivity) and clique number of graphs generated by a preferential attachment random graph model with an additional feature of allowing edge connections between existing vertices. Specifically, at each time step $t$, either a new vertex is added with probability $f(t)$, or an edge is added between two existing vertices with probability $1-f(t)$. We establish concentration inequalities for the global clustering and clique number of the resulting graphs under the assumption that $f(t)$ is a regularly varying function at infinity with index of regular variation~$-\gamma$, where $\gamma \in [0,1)$. We also demonstrate an inverse relation between these two statistics: the clique number is essentially the reciprocal of the global clustering coefficient. 

\end{abstract}
\maketitle


\section{Introduction}
Analyzing graph observables is crucial to understanding the topology and evolution of networks. These observable properties, including but not limited to node degree distribution, clustering coefficient, and clique number, provide invaluable insights into the network's overall structure and connectivity patterns.
The global clustering coefficient captures the density of triangles in the network and can reveal how nodes tend to form tightly-knit clusters. Whereas, the clique number, which is the size of the largest complete subgraph, can be thought of as a measure of how tightly connected the vertices in the graph are. 

The investigation and usage of the clustering coefficient and clique number are an interdisciplinary field that have appeared in a wide range of scientific papers. Articles from various fields have analyzed these properties of networks due to the fact that they provide insights into the underlying mechanisms governing network formation and evolution. In more theoretical fields, the study of these properties has led to the development of new graph theory concepts, see \cite{F07,K08} for a few examples of recent generalizations of clustering coefficients. In the field of chemistry, in \cite{R03} the authors employ cliques to describe chemicals within a chemical database that display a notable level of similarity to a target structure. In sociology, the study of these properties has helped to understand the formation of social networks and the spread of information within them, see \cite{F49,LP45}. Therefore, investigating the clustering coefficient and clique number is a crucial step towards developing a comprehensive understanding of network dynamics, and the interdisciplinary nature of these quantities underscores their significance.



In the classical preferential attachment model where new arriving vertices connect to $m$ neighbors \cite{BA99}, the clique number is at most $m$, and the clustering coefficient tends to $0$ as the inverse of the number of vertices in the graph \cite{bollrior}. In this paper, we focus on studying these two crucial statistics in graphs generated by a variation of this model, where new edges between existing vertices can be added at any time. We investigate the impact of this modified rule on these statistics. 

To facilitate our analysis and present our results clearly, we will first introduce a formal definition of the model in the following subsection.


\subsection{The model}
Firstly, let us introduce a new notation. Given a finite connected (multi)graph~$G$, let $\pi_G$ be the following distribution over the vertices of $G$ defined as follows
\begin{equation}\label{def:pi}
    \pi_G(v) := \frac{{\rm degree}(u)}{\sum_{w \in G} {\rm degree}(w)}.
\end{equation}
Our model generates a sequence of (multi)graphs in an inductive way, that is, we obtain the $t$-th graph by performing a graph operation on the $(t-1)$th graph. In our settings, there exist two possible ways of modifying a given graph $G$:
\begin{itemize}
    \item {\bf Vertex-operation}. Add a {\it new} vertex $v$ to $G$ and connect it to a vertex $u \in G$ selected accordingly to $\pi_G$;
    \item {\bf Edge-operation}. Add a new edge connecting two existing vertices of $G$, $u$ and $v$, both selected independently and accordingly to $\pi_G$. 
\end{itemize}
We highlight that in edge-operation, $u$ and $v$ might be the same vertex, in that case a loop is added to $v$, or they might be already connected in $G$, in which case a multiple edge is created between then.

The parameter of the model is a real function $f:(0,\infty) \to [0,1]$. For a fixed function $f$, we let $\{Z_t\}_{t \in \mathbb{N}}$ be an independent sequence of random variables such that $Z_t \sim {\rm Ber}(f(t))$. We start with an initial graph $G_0$ and then we generate the model inductively by obtaining~$G_{t+1}$ from $G_t$ as follows: 
\begin{itemize}
    \item {\bf Vertex-step.} If $Z_{t+1} = 1$, we perform a {\bf vertex-operation} on $G_t$ and set the resulting graph to be $G_{t+1}$. 
    \item {\bf Edge-step.} If $Z_{t+1} = 0$, we perform an {\bf edge-operation} instead.
\end{itemize}
In words, the model decides according to a coin, with a possibly time-dependent bias, which operation it is going to perform on the current graph in order to generate the next one.

The model here investigated was originally introduced in \cite{ARSEdge17} and, by choosing $f(t)\equiv 1$, one obtains the well-known Preferential Attachment (PA) random graph proposed by A-L. Barab\'asi and R. Albert in \cite{BA99} with parameter $m=1$. The case $f(t) \equiv p \in [0,1]$ is a particular case of the model introduced by C. Cooper and A. Frieze in \cite{CF03} and also investigated by F. Chung and L. Lu in \cite{CLBook}.

For the sake of readability and to avoid clutter we will make some choices notation-wise. Throughout the paper, we will use letters $f,g$ and $h$ to denote functions for which $G_t(f)$ makes sense, that is, $f,g,h: (0,\infty) \to [0,1]$. We will also simply write $G_t$ for the $t$-th graph generated by the above dynamics, whenever the function~$f$ is fixed and there is no risk of confusion. Alternatively, we will write $G_t(f)$ to highlight the dependence on $f$ when needed.


We will also set $G_0$ to be the multigraph on one vertex and a self-loop attached to it. The reader will notice  that the initial graph plays no important role in our results as long as it is finite.
And we will write $|G|$ as a shorthand for the cardinality of the vertex set of $G$. In this paper we investigate the aforementioned model when $f$ belongs to the family of regularly varying functions. We say a function $f:(0,\infty) \to (0,\infty)$ is a {\it regularly varying function at infinity} if there exists a real number $\gamma$ such that for all $a>0$ the following identity is satisfied
\begin{equation}\label{def:rv}
    \lim_{t\to \infty}\frac{f(at)}{f(t)} = a^\gamma.
\end{equation}
The constant $\gamma$ is called the {\it index of regular variation}. Sometimes, when $\gamma=0$, we may call the function $f$ {\it slowly varying}.


\subsection{Main results: Statements and Discussion} 
In this part we will state and discuss our main results. The results concern concentration inequalities for the clique number and global clustering coefficients. Our results are particularly interesting as they provide a characterization of such statistics in terms of the index of regular variation when $f$ is taken as a regularly varying function. We also prove a continuity theorem, meaning that under a suitable notion of convergence for the function $f$, we have convergence in total variation distance of the respective processes.



\subsubsection{Cliques} A complete subgraph is a subgraph in which every pair of vertices is adjacent to each other. The clique number of a graph is the size of the largest complete subgraph (clique) that can be found in the graph. Again for the sake of readability, we will denote the clique number of $G_t(f)$ simply by $\omega_t$ whenever $f$ is fixed or understood from the context.

The most straightforward problem about clique number concerns finding its order of magnitude. In \cite{BCH20}, K. Bogerd, R. M. Castro and R. van der Hofstad, prove a concentration result for the clique number in rank-1 random graphs. Another interesting problem is the one of detecting artificially planted cliques, that is, by observing realizations of the model, decides whether it contains a clique that has been artificially added to the graph. This type of problem is investigated by K. Bogerd, R. M. Castro, R. van der Hofstad and N. Verzelen in \cite{BCHV21}.

In our context, we establish two results involving the clique number of $G_t(f)$. Firstly, we have proved a general lower bound for the clique number of $G_t$ that holds under the mild assumption that~$f$ decreases to zero.
\begin{theorem}[Existence of Large cliques]\label{t:largecliques}
	If $f$ decreases to zero, then for every~$\delta\in(0,1)$, there exist positive constants $c_1$ and $c_2$ depending on $\delta$ and $f$ only such that
\[
\Pd\left(  \omega_t \ge \frac{c_2\Ed\left[|G_{t^{(1-\delta)/2}}|\right]}{\log t}\right) \ge  1-\frac{2}{\log t} - \exp \Lbrace - \frac{c_2\Ed\left[|G_{t^{(1-\delta)/2}\log t}|\right]}{4\log t}\Rbrace,
\]
\end{theorem}
The above theorem yields an interesting result: for certain regimes of the model, it states that~$G_t(f)$ contains a clique of order approximately $\sqrt{|G_t(f)|}$. For example, consider the case where $f(t) = t^{-\gamma}$ with $\gamma \in (0,1)$. This gives us~$\Ed\left[|G_t(f)|\right] \approx t^{1-\gamma}$. Thus, by Theorem \ref{t:largecliques} we obtain, with high probability, a clique of order at least~$t^{(1-\gamma)(1-\delta)/2}$ which is approximately   $\left( \Ed\left[|G_t(f)|\right]\right)^{(1-\delta)/2}$.

As our second main result, we have also provided a comprehensive characterization of this statistic along with the global clustering coefficient under the additional assumption of $f$ being regularly varying. In other words, we establish that for the vast and important class of regularly varying functions at infinity the lower bound given by Theorem \ref{t:largecliques} is sharp up to factors of order $t^{\delta}$ for any $\delta$.  This is our next result.
\begin{theorem}[Clique number for RV functions]\label{thm:clique}Assume that $f$ decreases to zero and it is a regularly varying function with index of regular variation $-\gamma$,  for $\gamma \in [0,1)$. Then, for every~$\delta \in (0, 1)$, there exist positive constants $c_1,c_2$ and $c_3$ depending on $\delta$ only such that
    $$
    \Pd\left( c_1t^{(1-\gamma)(1-\delta)/2} \le \omega_t \le c_2t^{(1-\gamma)(1+\delta)/2}\right) \ge 1 -\frac{c_3}{\log t}.
    $$
\end{theorem}
From a computational perspective, the problem of finding cliques in a graph are not only too time-consuming, but some are NP-complete. In our case, as a consequence of our approach, the clique given by Theorem \ref{thm:clique} is likely to be formed by the $t^{(1-\gamma)(1+\delta)/2}$ first vertices added by the process. 

By Theorem \ref{thm:clique} and properties of the random sequence $\{\omega_t\}_{t \ge 0}$ we prove a convergence result
\begin{corollary}[Convergence of clique number]\label{cor:cliquesRV} Assume that $f$ decreases to zero and it is a regularly varying function with index of regular variation $-\gamma$,  for $\gamma \in [0,1)$. Then,
		\[
		\lim_{t \to \infty }\frac{\log \omega_t}{\log t } = \frac{1-\gamma}{2}, \; \text{almost surely.}
		\]

\end{corollary}
In \cite{ARS21C} the authors prove the convergence of the clique in the $\log$ scale for $f(t) \equiv p$, with~$p\in [0,1]$. Their result states that 
\[
		\lim_{t \to \infty }\frac{\log \omega_t}{\log t } = \frac{1-p}{2-p}, \; \text{almost surely.}
		\]
Then Corollary \ref{cor:cliquesRV} can be seen as some sort of limit of the above convergence when the parameter $p$ goes to zero. From this perspective, if the parameter~$p$ goes to zero as a slowly varying function, then one obtains the expected $\frac{1}{2}$ as the limit. However, if $p$ goes too fast to zero, then a limit smaller than $\frac{1}{2}$ is obtained.

\subsubsection{Clustering coefficient}
The global clustering coefficient is a statistic that measures the extent to which vertices in a graph tend to cluster together. It is defined as the ratio of the number of triangles (complete subgraphs of size three) in the graph to the total number of connected triples of vertices (a.k.a cherries).

Due to the fact that in our settings $G_t(f)$ might be a multigraph with multiple edges or self-loops, we will compute the global clustering for the {\it simplification} of~$G_t(f)$, that is, given~$G_t(f)$, we obtain a simple graph $\widehat{G}_t(f)$ obtained from~$G_t(f)$ by removing multiple edges and loops.

Given a simple graph $H$ we will let $\Delta(H)$ be the number of triangles in $H$. Whereas, $\Lambda(H)$ denotes the number of cherries (paths of length two) in $H$. In the case of a multigraph $G$, we will simply write $\Delta(G)$ and $\Lambda(G)$ meaning $\Delta(\widehat{G})$ and $\Lambda(\widehat{G})$, respectively. That is, $\Delta(G)$ and $\Lambda(G)$ count the number of triangles and cherries respectively without multiplicities. Then, we define the global clustering coefficient of $G_t(f)$ as 
\begin{equation*}
    c^{\rm glob}_t = c^{\rm glob}(G_t(f)) := 3\times \frac{\Delta(G_t(f))}{\Lambda(G_t(f))}.
\end{equation*}
B. Bollob\'{a}s and O. Riordan studied the BA-model with $m\ge 2$ in \cite{bollrior}. In this model, a new vertex with $m$ randomly chosen neighbors is added to the graph at each step. The authors proved that $\mathbb{E}[c^{\rm glob}_t]$ decreases at the rate of $\log^2(t)/t$, and the expected number of triangles at time $t$ is roughly $\log^3(t)$. N. Eggemann and S. D. Noble studied a variation of this model, called the affine version, where vertices are selected based on their degree and a constant of attractiveness $\delta$ in \cite{EN11}. They proved that, for any positive $\delta$, the expected number of triangles is of order $\log^2(t)$, and $\mathbb{E}[c^{\rm glob}_t]$ decays as $\log(t)/t$.

In \cite{JM15}, E. Jacob and P. M\"{o}rters investigated the global clustering for a different model which also takes into account spatial proximity of vertices to create connections. Their study found that $c^{\rm glob}_t$ converges in probability to a constant whose positivity depends on the choice of parameters. In \cite{SHL19}, C. Stegehuis, R. van der Hofstad and J. Leeuwaarden, also investigate clustering for random graph models with latent geometry.

In our study, we establish a concentration inequality result for $c^{\rm glob}_t$ that shows that in our case, $c^{\rm glob}_t$ decays polynomialy fast and the exponent depends on the index of regular variation of $f$. Additionaly, we show an almost surely convergence of $c^{\rm glob}_t$ in the right scale.

\begin{theorem}[Clustering coefficient for RV functions]\label{thm:clustering}Assume that $f$ decreases to zero and it is a regularly varying function with index of regular variation $-\gamma$,  for $\gamma \in [0,1)$. Then, for every $\delta \in (0, 1)$, there exist positive constants $c_1$ and $c_2$ depending on $\delta$ only such that
    $$
    \Pd\left( c_1t^{-(1-\gamma)(1+\delta)/2} \le c^{\rm glob}_t \le c_2t^{-(1-\gamma)(1-\delta)/2}\right) \ge 1 -\frac{1}{\log t}.
    $$
    
\end{theorem}

The bounds and rates given by the above theorem are good enough to show the convergence of the clustering coefficient at the $\log$ scale.
\begin{corollary}[Convergence of clustering coefficient]\label{cor:clusteringRV} Assume that $f$ decreases to zero and it is a regularly varying function with index of regular variation $-\gamma$,  for $\gamma \in [0,1)$. Then,
    \[
    \lim_{t \to \infty }\frac{\log c^{\rm glob}_t}{\log t } = -\frac{1-\gamma}{2}, \; \text{almost surely.}
    \]

\end{corollary}
In \cite{ARS21C} the authors prove that for $f(t)\equiv p$, for $p\in [0,1]$, the following limit holds
$$
\lim_{t \to \infty }\frac{\log c^{\rm glob}(G_t(f))}{\log t } = \frac{3(1-p)}{2-p} -2+p,\; \text{almost surely.}
$$
Similarly to the case of the clique number, the case in which $f$ is a slowly varying function that decreases to zero can be seen as the limit case when $p$ goes to zero. 

Observe that Theorem \ref{thm:clustering} combined with Theorem \ref{thm:clique} shows a surprisingly intriguing reciprocal relation between global clustering and clique number. We state this relation in the result below.
\begin{corollary}[Inverse Relation]\label{cor:inverserel}Assume that $f$ decreases to zero and it is a regularly varying function with index of regular variation $-\gamma$,  for $\gamma \in [0,1)$. Then, for every $\delta \in (0, 1)$, there exist positive constants $c_1$ and $c_2$ depending on $\delta$ only such that
    $$
    \Pd\left( \frac{c_1}{\omega(G_t(f))^{1+\delta}} \le c^{\rm glob}(G_t(f)) \le \frac{c_2}{\omega(G_t(f))^{1-\delta}}\right) \ge 1 -\frac{1}{\log t}.
    $$
    
\end{corollary}

\subsubsection{Comparing $G_t(f)$ and $G_t(h)$ }
If two edge-step functions are ``close" to each other in some sense, then one should expect that the random processes they generate should be ``close" as well. In order to formalize this intuition, we have developed a grand coupling mechanism that enables us to generate all processes $\{G_t(f)\}_t$, for any $f$, from the same source of randomness. The coupling allows us to compare the statistical properties of~$G_t(f)$ and $G_t(h)$ by simply comparing the functions~$f$ and $h$. This can be seen as a {\it Transfer Principle} result that allows us to transfer results for $G_t(f)$ to $G_t(h)$ under some assumptions on the relation between $f$ and $h$. 

The idea that $G_t(f)$ is close to $G_t(h)$ whenever $f$ and $h$ are close is formalized in the theorem below, which is a consequence of our Grand Coupling introduced at Section \ref{s:randomtree}. In words, if $f$ and $h$ are close in the~$L_1(\N)$-norm (denoted by $\|\cdot \|_1$), then $\mathrm{Law}(\{G_t(f)\}_{t\ge 1})$ is close to $\mathrm{Law}(\{G_t(h)\}_{t\ge 1})$ in the \textit{total variation distance}, denoted by $\mathrm{dist}_{TV}(\cdot, \cdot)$. 
\begin{theorem}\label{t:dtv}
	Consider any two $f$ and $h$. We have
	\begin{equation}
	\label{eq:dtv}
	\dist_{TV}\left(\mathrm{Law}(\{G_t(f)\}_{t\geq 1}),\mathrm{Law}(\{G_t(h)\}_{t\geq 1})\right)\leq \| f- h\|_1.
	\end{equation}
	In particular, the following implication holds
	\begin{equation}
	\label{eq:dtv2}
	 f_n \stackrel{L_1(\N)}{\longrightarrow}  f \implies \mathcal{L}((G_t(f_n))_{t\geq 1})\stackrel{\mathrm{dist}_{TV}}{\longrightarrow} \mathcal{L}((G_t(f))_{t\geq 1}).
	\end{equation}
\end{theorem}
The theorem presented above represents an interesting perturbative statement, illuminating the ability to introduce a small amount of noise $\epsilon = \epsilon(t)$ into the function $f$, while preserving the integrity of the original process up to a maximum error of $| \epsilon |_1$ in total variation distance. This result affirms the robustness of the process to external perturbations.

One difficulty in our setup is the degree of generality we work with. Our case replaces the parameter $p \in (0,1]$ in the models investigated in \cite{ARS17, CLBook, CF03} by any non-negative real function~$f$ with $\| f \|_{\infty} \le 1$. The introduction of such function naturally increases the complexity of any analytical argument one may expect to rely on and makes it harder to discover \textit{threshold} phenomena. 
On the other hand, the fact that we are working under such large space of parameters is also one of the reasons why results like Theorems \ref{thm:clique} and \ref{thm:clique} or even the Corollaries \ref{cor:cliquesRV} and \ref{cor:clusteringRV} are surprising, because they depend on a simple functional of the parameter $f$: its index of regular variation at infinity. Under such wide degree of freedom, there is no reason to expect such simple characterizations of statistics to hold true.

\subsection{Main technical ideas}
In order to prove Theorem~\ref{t:largecliques} and Theorem \ref{t:dtv}, we construct an auxiliary process that we call the \textit{doubly-labeled random tree process}, $\{\cal{T}_{t}\}_{t \ge 1 }$. In essence, this process is a realization of the traditional BA-model (obtained in our settings choosing $f \equiv 1$) where each vertex has two labels attached to it. We then show in Proposition~\ref{prop:gcoupling} how to generate~$G_t(f)$ from $\cal{T}_t$ using the information on those labels. This procedure allows us to generate $G_t(f)$ and $G_t(h)$, for two distinct functions~$f$ and~$h$, from exactly the same source of randomness. The upshot is that the parameter function $f$ may be seen as a map from the space of doubly-labeled trees to the space of (multi)graphs -- therefore it makes sense to use the notation~$f(\cal{T}_t)$. Furthermore, this map has the crucial property of being \textit{monotonic} (in a way we make precise latter). Roughly speaking, if $f \le h$, then certain monotonic statistics respect this order, so if~$\zeta$ is such an statistic, then $\zeta(f(\cal{T}_t)) \le \zeta(h(\cal{T}_t))$. In Proposition~\ref{p:diammaxdegProp} we give important examples of suitable monotonic statistics. The main advantage of our machinery is that it allows us to transpose some results about graphs generated for functions in a particular regime to another just by comparing the functions themselves. This important feature of our approach allows us, for instance, to use known results about cliques when $f$ is taken to be a constant less than one to propagate this result down other regimes of functions.

While the coupling approach offers a clear pathway for extending results of one regime of~$f$ to another, some statistics still demand further work. As for an upper bound for the clique number, we obtain it by controlling the number of triangles. In this approach, we investigate the simplification of $G_t$ denoted by $\widehat{G}_t$. By providing a sharp upper bound on the number of edges of $\widehat{G}_t$, we can use deterministic results to bound the number of triangles, which by its turn can provide an upper bound for the order of the largest clique. The upper bound on the number of edges in $\widehat{G}_t$ requires a sharp upper bound for vertices degree and for the number of edges between two vertices. The upper bound on the number of edges between vertices is obtained using a method inspired by the martingale methods used in the analysis of vertices' degree in the literature of preferential attachment random graphs.

\section{A Grand Coupling: The doubly-labeled random tree process}\label{s:randomtree}
In this section we introduce a stochastic process~$\{\mathcal{T}_t\}_{t\geq 1}$ that provides a grand coupling between the random graphs~$\{G_t(f)\}_{t\geq 1}$ for every possible choice of~$f$.

The process $\{\mathcal{T}_t\}_{t\geq 1}$ is essentially a realization of the Barab\'asi-Albert random tree with the distinction that each vertex has two labels: an earlier vertex chosen according to the preferential attachment rule and an independent uniform random variable. The label consisting in the earlier vertex can be seen as a ``ghost directed edge'', we later use these random labels to collapse subsets of vertices into a single vertex in order to obtain a graph with the same distribution as $G_t(f)$ for any prescribed function~$f$. 

We begin our process with a graph~$\mathcal{T}_1$ consisting as usual in a single vertex and a single loop connecting said vertex to itself. We then inductively construct the labeled graph~$\mathcal{T}_{t+1}$ from~$\mathcal{T}_t$ in the following way:
\begin{figure}[ht]\label{f:gtildedef}
\centering
\includegraphics[scale = .63]{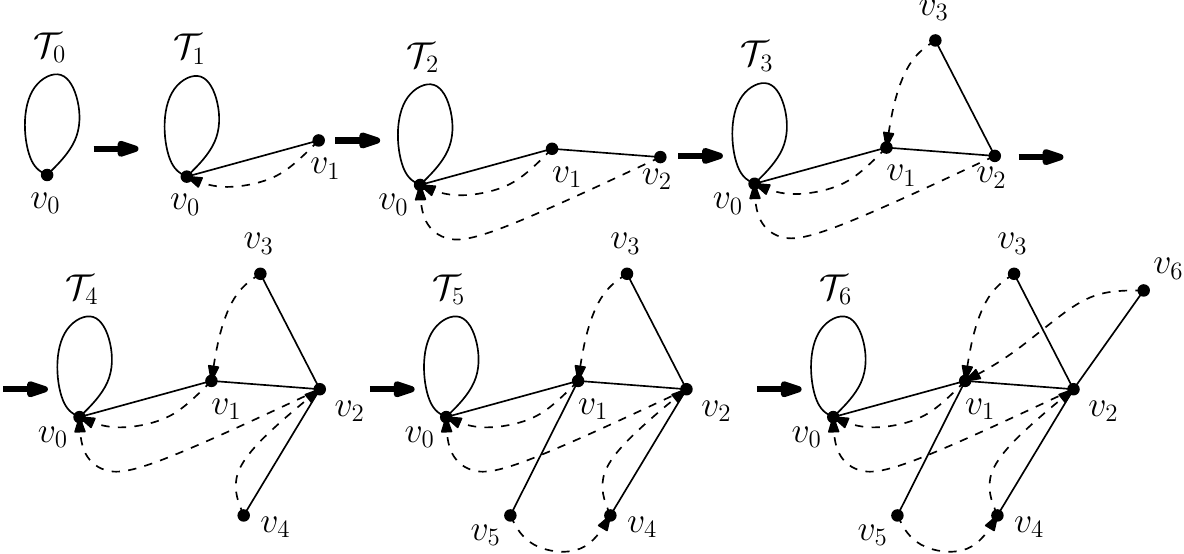}
\vspace{0.2cm}
\caption{A sample of the process~$\{\mathcal{T}_t\}_{t\geq 1}$ up to time~$6$ without the uniform labels. The dashed lines indicate the label~$\ell(v_j)$ taken by each vertex~$v_j$.}
\end{figure}
\begin{itemize}
\item [(i)] We add to~$\mathcal{T}_t$ a vertex~$v_{t+1}$;

\item [(ii)] We tag~$v_{t+1}$ with a random label~$\ell(v_{t+1})$ chosen from the set~$V(\mathcal{T}_t)$ with the preferential attachment rule, that is, with probability of choosing~$u\in V(\mathcal{T}_t)$ proportional to the degree of~$u$ in~$\mathcal{T}_t$;

\item [(iii)] Independently from the step above, we add an edge~$\{w,v_{t+1}\}$ to~$E(\mathcal{T}_t)$ where~$w\in V(\mathcal{T}_t)$ is also randomly chosen according to the preferential attachment rule, see Figure~\ref{f:B17}.
\end{itemize}
We then finish the construction by tagging each vertex~$v_j\in V(\mathcal{T}_t)$ with a second label consisting in an independent random variable~$U_j$ with uniform distribution on the interval~$[0,1]$, as shown in Figure~\ref{f:B17}. We note that only the actual edges contribute to the degree taken in consideration in the preferential attachment rule, the tags are not considered.

We now make precise the notation $f(\cal{T}_t)$ which indicates that a function $f$ may also be seen as a function that maps a doubly-labeled tree to a (multi)graph. Our goal is to define this map in such way that $f(\mathcal{T}_t) \stackrel{d}{=} G_t(f)$.

In order to do so, let us fix such a function~$f$. Given~$v_j\in V(\mathcal{T}_t)$, we compare~$U_j$ to~$f(j)$. If~$U_j\le f(j)$, we do nothing. Otherwise, we collapse~$v_j$ onto its label~$\ell(v_j)$, that is, we consider the set~$\{v_j,\ell(v_j)\}$ to be a single vertex with the same labels as~$\ell(v_j)$. We then update the label of all vertices~$v$ such that~$\ell(v)=v_j$ to~$\{v_j,\ell(v_j)\}$. This procedure is associative in the sense that the order of the vertices on which we perform this operation does not affect the final resulting graph, as long as we perform it for all the vertices of~$\mathcal{T}_t$. We let then $f(\cal{T}_t)$ be the (multi)graph obtained when this procedure has run over all the $t$ vertices of $\cal{T}_t$. We refer to Figure~\ref{f:B17} as an illustration of the final outcome. 

In the next proposition we prove that $f(\cal{T}_t)$ is indeed distributed as $G_t(f)$. 
  \begin{figure}
    \centering
    \begin{subfigure}{0.4\textwidth}\label{f:gcaldef}
        \centering
        \includegraphics[width=\textwidth]{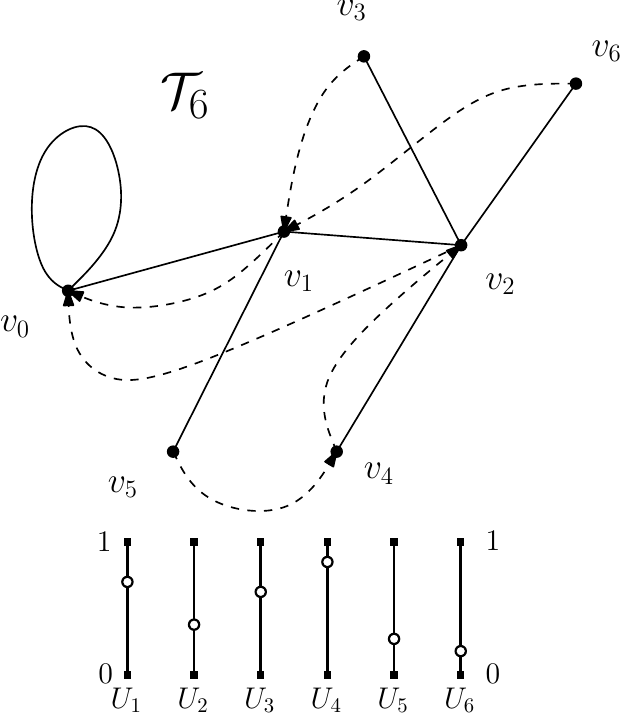}
    \end{subfigure}
    \begin{subfigure}{0.4\textwidth}\label{f:gcaldef}
        \centering
        \includegraphics[width=\textwidth]{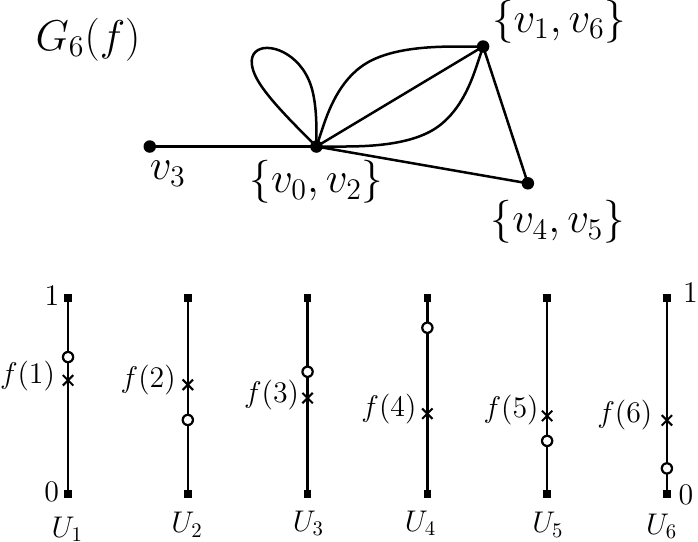}
    \end{subfigure}
    \caption{On the left, the graph~$\mathcal{T}_6$, where each vertex receives an independent uniform random variable label. On the right, how one can sample the distribution of~$G_6(f)$ using the~$\mathcal{T}_6$.}
    \label{f:B17}
  \end{figure}
  
\begin{proposition}
\label{prop:gcoupling}
Let~$\cal{T}_t$ be the doubly-labeled tree above defined. Then,
\[
f(\mathcal{T}_t) \stackrel{d}{=} G_t(f).
\]
\end{proposition}
\begin{proof}
We first observe that the associativity of the collapsing operation and the independence of the sequence $\{U_j\}_{j\geq 1}$ from the previous operations imply that we can glue together the vertex $v_{t+1}$ to $\ell(v_{t+1})$ whenever $U_{t+1}>f(t+1)$ right after we complete step~(iii) of the above construction by induction. The resulting graph has either a new vertex~$v_{t+1}$ with an edge~$\{w(v_{t+1}),v_{t+1}\}$ or an edge~$\{w(v_{t+1}),\ell(v_{t+1})\}$ with the exact same probability distribution as the~$(t+1)$-th step in the construction of the graph~$(G_t(f))_{t\geq 1}$. By induction, both random graphs have the same distribution.
\end{proof}

\section{Applications of The Grand Coupling}
In the light of the Section \ref{s:randomtree} and Proposition~\ref{prop:gcoupling}, throughout this section, we will tacitly assume that all processes of the form $\{G_t(f)\}_{t\ge 1}$ for all possible functions $f$ are on the same probability space provided by our previous results.

A straightforward consequence of Proposition~\ref{prop:gcoupling} is Theorem~\ref{t:dtv} which roughly speaking states that functions close to each other in the $L_1(\N)$-norm generate essentially the same processes. Once we have the machinery provided by Proposition~\ref{prop:gcoupling}, the proof of this fact becomes a simple application of the union bound.
\begin{proof}[Proof of Theorem~\ref{t:dtv}] Fixed two possible functions $f$ and $h$, by Proposition~\ref{prop:gcoupling} we have that
	\begin{equation}
	\mathrm{Law}(\{G_t(f)\}_{t\geq 1}) = \mathrm{Law}(\{f(\cal{T}_t)\}_{t\geq 1}).
	\end{equation}
Thus
	\begin{equation}
	\begin{split}
	\dist_{TV}\left(\mathrm{Law}(\{G_t(f)\}_{t\geq 1}),\mathrm{Law}(\{G_t(h)\}_{t\geq 1})\right)
	&\leq\Pd\left( \{f(\cal{T}_t)\}_{t\geq 1} \neq \{h(\cal{T}_t)\}_{t\geq 1}     \right)
	\\
	&\leq \sum_{i=1}^\infty \Pd\left(    U_i   \in (f(i)\wedge h(i),f(i)\vee h(i)  )  \right)
	\\
	&=
	\sum_{i=1}^\infty | f (i)- g (i)|
	\\
	&= \| f- g\|_1,
	\end{split}
	\end{equation}
	and the above equation immediately implies~\eqref{eq:dtv2}.
\end{proof}
\subsection{Monotone Statistics}
The next step is to use the coupling provided by the doubly-labeled random tree process to compare statistics of graphs generated by different functions. This method will allow us to transport results we have obtained for a fixed~$f$ to other edge-step functions by comparing them with $f$. To do this we introduce the notion of \textit{increasing (decreasing)} graph observable.
\begin{definition}[Monotone Statistics]\label{def:graphobs} We say that a (multi)graph statistic $\zeta$ is \textit{increasing} if, given two edge-step functions $f$ and $h$, we have that 
	\begin{equation}\label{eq:incrob}
		f(s) \le h(s), \forall s \le t \implies	\zeta(f(\cal{T}_t)) \le \zeta(h(\cal{T}_t)) \; \textit{a.s.}
	\end{equation}
When the second inequality in (\ref{eq:incrob}) holds with ``$\ge$", we say $\zeta$ is decreasing.
\end{definition}
A first example of an increasing statistic is the total number of vertices. Indeed, since every vertex~$v_j$ of $\mathcal{T}_t$ remain preserved under $f$ whenever its $U_j$ label is less than $f(j)$, it follows that every $f$-preserved $v_j$ is $h$-preserved as well.

The next proposition states that the maximum degree is a decreasing statistic whereas the diameter is an increasing one.
\begin{proposition}\label{p:diammaxdegProp} Let $f$ and $h$ be two functions satisfying the relation $f(s)\le h(s)$ for all $s \in (1, \infty)$. Then,
\begin{enumerate}[(a)]
	\item the maximum degree is a \emph{decreasing }graph statistic. More precisely, if we let $D_{\textit{max}}(G)$ denote the maximum degree of a (multi)graph $G$, then
	\[
	\Pd \left(\forall t \in \N, \; D_{\textit{max}}(G_t(f)) \ge D_{\textit{max}}(G_t(h))\right) = 1;
	\]
	\item the diameter is an \emph{increasing }graph statistic. More precisely,
	\[
	\Pd \left(\forall t \in \N, \; \diam(G_t(f)) \le \diam(G_t(h))\right) = 1
	\]
\end{enumerate}

\end{proposition}
\begin{proof} The proofs follow from an analysis of the action of $f$ over $\cal{T}_t$.

\noindent \underline{Proof of part (a):}
Given a vertex $v \in f(\cal{T}_t)$ we let $\mathsf{solid}$-$\mathsf{neighbors}(v)$ be the set of its neighbors in $\cal{T}_t$ (that is, vertices sharing a solid edge with it, which would come from a vertex-step). We also define the set $\mathsf{dashed}$-$\mathsf{neighbors}(v)$ as the vertices $v_s$ in $\cal{T}_t$ with label $\ell$ pointing to $v$ and whose corresponding uniform label $U_s$ is larger than $f(s)$. To find the degree $D_t(v)$ of $v$ in $f(\cal{T}_t)$, one performs the following computation:
\begin{itemize}
    \item initialize $D_t(v)$ as the cardinality of $\mathsf{solid}$-$\mathsf{neighbors}(v)$ and the set $\mathsf{to}$-$\mathsf{explore}$ as $\mathsf{dashed}$-$\mathsf{neighbors}(v)$;
    \item sequentially in some arbitrary order, for every $w \in \mathsf{to}$-$\mathsf{explore}$, sum the cardinality of $\mathsf{solid}$-$\mathsf{neighbors}(w)$ to $D_t(v)$ and add the vertices of $\mathsf{dashed}$-$\mathsf{neighbors}(w)$ to $\mathsf{to}$-$\mathsf{explore}$;
    \item repeat the above until no more vertices are in $\mathsf{to}$-$\mathsf{explore}$.
\end{itemize}
In other words, given the directed tree with $v$ as the root, and edges coming from labels that become edge-steps after the action of $f$, we perform a graph-search algorithm, like breadth-first search, and sum the $\mathcal{T}_t$-degree of each queried vertex. Note that each vertex of $\cal{T}_t$ can be accounted for multiple times in this procedure when we sum the degree in $\cal{T}_t$. This algorithm does yield the degree of vertex $v$, as there is a bijection between each half-edge incident to $v$ in $f(\cal{T}_t)$ and the multiset of vertices of $\cal{T}_t$ in $\mathsf{solid}$-$\mathsf{neighbors}(w)$ as we explore vertices $w \in \mathsf{to}$-$\mathsf{explore}$.

Now we note that if $f(s) \leq h(s)$ for all $s$, the $\mathsf{dashed}$-$\mathsf{neighbors}$-directed tree associated to the algorithm that discovers the degree of a vertex $v$ in $f(\cal{T}_t)$ contains the tree that discovers the degree of the same vertex in $h(\cal{T}_t)$. So if $v_s$ is the vertex with maximum degree in $h(\cal{T}_t)$, and its associated uniform is larger than $f(s)$, by direct comparison we obtain the result. If not, then this means that $v_s$ will be collapsed into some vertex $w \in f(\cal{T}_t)$, and the tree from the degree-discovering algorithm for this vertex will contain the tree associated to $v_s \in h(\cal{T}_t)$. This finishes the proof of part (a).

\noindent \underline{Proof of part (b):} Observe that if $u$ and $w$ are vertices whose $U$-label is less than $f$ and such that $u \leftrightarrow w$ in $G_t(f)$, then $u$ and $w$ are possibly not connected by a single edge in $G_t(h)$, in other words, if $\dist_{G_t(f)}(u,w) = 1$ then $\dist_{G_t(h)}(u,w) \ge 1$. Therefore, if $v_0 =u, v_1, \cdots, v_k = w$ is a minimal path in $G_t(f)$, by the previous observation, this path induces a path in $G_t(h)$ whose length is equal to or greater than $k$.

On the other hand, if $u=v_0 , v_1, \cdots, v_k = w$ is a path realizing the minimal distance between~$u$ and~$w$, in $G_t(h)$ we have that this path induces another path in $G_t(f)$ of same length or less. To see this, just notice that, if $U_j > f(j)$ and $\ell(v_j)$ points to a vertex outside the path for some $j \in \{2,3,\cdots,k-1\}$, then, after identifying $v_j$ to its $\ell$-label, this operation does not increase the path length. This is enough to conclude the proof.
\end{proof}
Having small diameter, typically $\log$ of the number of nodes, is a feature shared by several real-life networks, as observed by S. Strogatz and D. J. Watts in \cite{SW98}, where the coined the term `small-world' networks. Thus investigating the order of the diameter is vital for models whose purpose might be modeling real-life networks. 

In \cite{ARS20}, the authors showed a characterization of the diameter of $G_t(f)$ for regularly varying function~$f$ in terms of the index of regular variation. The proposition below offers a universal upper bound that guarantees that $G_t(f)$ is small-world for any choice of $f$.
\begin{proposition}[Universal Upper bound for the Diameter] There exists a universal constant $c$ such that 
	$$
	\Pd \left(\diam(G_t(f)) \le c\log t \right) = 1- o(1).
	$$
\end{proposition}
\begin{proof}
The proof follows from item (b) of Proposition \ref{p:diammaxdegProp} and known results for the classical preferential attachment model. Indeed, in \cite{BR04}, B. Bollob\'{a}s and O. Riordan proved that for $f \equiv 1$ the diameter of $G_t(f)$ is bounded from above by $c\log t$ asymptotically almost surely. Since $f$ is bounded by $1$, a combination of this result with item (b) of Proposition \ref{p:diammaxdegProp} proves the result.
\end{proof}

\section{Proof of Theorem \ref{t:largecliques}: existence of large cliques}
In this section we keep exploring the machinery developed in previous sections. Here, we transfer the existence of large cliques using the function itself, i.e., if $h$ is such that~$G_t(h)$ has a clique of order $K$ and $f$ another function satisfying $f(s) \le h(s)$ for all $s \in (1, \infty)$, then the clique existence propagates to $G_t(f)$ too but with a possibly $f$-dependent order. 

We will let $G_t(p)$ be the graph generated by $f\equiv p \in (0,1)$ and we will let $\omega_t(p)$ be the clique number of $G_t(p)$. And for the sake of convenience, we state the result about $\omega_t(p)$.
\begin{theorem}[Theorem~$1$ from~\cite{ARS17}]\label{teo:community} For any $\e>0$ and every $p \in (0,1)$, it follows that
	\begin{equation}
	\Pd\left( \omega_t(p) \ge t^{(1-\e)\frac{1-p}{2-p}}\right) = 1-o(1).
	\end{equation}
\end{theorem}
We can now use the doubly-labeled tree together with the above results in order to prove Theorem~\ref{t:largecliques}. In order to avoid clutter, let us introduce a new notation for the expected number of vertices,
\begin{equation}
F(t) := \Ed [|G_t(f)|] = 1+ \sum_{s=2}^{t}f(s).
\end{equation}
In order to formalize the strategy of transporting a clique in $G_t(p)$ down to $G_t(f)$ it will be important to control the time the $j$-th vertex of $G_t(p)$ is added. In our proof, we essentially condition on the existence of a large clique in $G_t(p)$ and then using the independence of the $U$-labels we pay a $f$-dependent price to guarantee that part of the cliques' vertices are also a vertex in $G_t(f)$. However, if the vertices in the clique took too long to be added, then the $f$-price we pay to force them to be vertices in $G_t(f)$ as well is too expensive. This would result in a much smaller clique in $G_t(f)$.

Consequently, it will be important to use the fact that for the process $\{G_t(p)\}_{t\ge 0}$, does not take too much time between vertices addition, since the time interval between the addition of any two vertices follows a geometric distribution of parameter $p$. We can formalize this by defining inductively the following sequence of stopping times, let $\tau_0(p) \equiv 0$ and define $\tau_k(p)$ as below
\begin{equation*}
    \tau_k(p) := \inf \{ n \ge \tau_{k-1}(p) \; : \; U_n \le p \}.
\end{equation*}
Then, noticing that $\tau_k(p) - \tau_{k-1}(p) \sim Geo(p)$ for any $k>0$ we have that
\begin{equation}\label{eq:bounddeltatauk}
    \Pd\left( \tau_k(p) - \tau_{k-1}(p) \ge \frac{10\log t}{\log(1/(1-p))}\right) \le t^{-10}.
\end{equation}
The above bound will give us control over the time vertices have been added by the process~$\{G_t(p)\}_{t\ge 0}$. With the discussed above we prove the main result of this section.
\begin{proof}[Proof of Theorem \ref{t:largecliques}]
Given~$\delta>0$ sufficiently small, We begin by letting $\varepsilon$ be small enough and setting $p=\delta/6$ so that
\begin{equation*}
\alpha=\alpha(p,\varepsilon):=\frac{1-p}{2-p}(1-\varepsilon)>\frac{1}{2}\left(1-\frac{\delta}{2}\right).
\end{equation*}
In the proof of Theorem~$1$ of~\cite{ARS17} one uses Theorem~$2$ of~\cite{ARS17} to show that for any $\varepsilon'\in(0,\alpha)$ there exists a fixed integer~$m=m(\e,\e',p)>0$  with the following property: if one divides the set of vertices added by the process $\{G_t(p)\}_{t\ge 1}$ up to time~$t$ into disjoint subsets of~$m$ vertices born \emph{consecutively}, then with high probability (at least $1$ minus a polynomial function of~$t$) from the $t^{\varepsilon'}$-th subset to the $t^{\alpha}/m$-th one, we can select one vertex in each subset in such a way that the subgraph induced by the set of chosen vertices is a complete subgraph of~$G_t(p)$. These vertices are the \textit{leaders} of their block.

Let $\mathcal{K}$ be the random set formed by the leader vertices of the $t^{\varepsilon'}$-th to the $t^{\alpha}/m$-th blocks of $m$ vertices, which form a clique in $G_t(p)$ w.h.p. Observe that in the event these vertices do not form a clique or do not exist $t^{\alpha}/m$ blocks of $m$ vertices the random set $\mathcal{K}$ is empty.

Now, let $N$ be the following random variable
\begin{equation}
    N := \sum_{v_j \in \cal{K}} \mathbb{1} \Lbrace U_j \le f(j)\Rbrace.
\end{equation}
In words, $N$ counts how many leading vertices become a vertex of $G_t(f)$ as well. Here, we adopt the convention that a sum over an empty set of indexes is zero.

To avoid clutter, let us write $k = t^{\alpha}/m$. Given a set of vertices $S = \{v_{t_1}, \dots, v_{t_k}\}$ in $G_t(p)$ the event 
$\{\mathcal{K} = S\}$
then means the vertices $v_{t_j}$, for $j \in \{1,\dots, k\}$ are all vertices of $G_t(p)$ and they are all connected between themselves in $G_t(p)$. Thus, due to independent nature of the $U$-label assignments in the doubly-labeled random tree process the following identify holds
\begin{equation}\label{eq:utj}
    \Pd \left( U_{t_j} \le f(t_j)\mid  \cal{K} = S\right) = p^{-1}f(t_j).
\end{equation}
Additionally, we call $S$ \textit{good} if 
$$
v_{t_j} \in S \implies t_j \le c_1j \cdot m \cdot \log t,
$$
where $c_1 = 10\log(1/(1-p))$. In words, a good set of vertices in $G_t(p)$ are those whose vertices did not take too much time to be added by the process $\{G_t(p)\}_{t \ge 0}$. Observe that by~\eqref{eq:bounddeltatauk} and union bound over the number of vertices yields
\begin{equation}
    \Pd \left( \exists S \subset G_t(p)\; : \; S \text{ is bad }\right) \le t^{-9}.
\end{equation} 

Having the above in mind and letting $m$ be the auxiliary $m$ that appears in proof of Theorem~1 of \cite{ARS17}, define the following event
\begin{equation}
A_{p, \varepsilon', m, k} := \left \lbrace 
\begin{array}{c}
\exists \{s_1, \cdots, s_{k}\} \subset \{t^{\varepsilon'}, \dots, c_1m(t^{\alpha} + t^{\varepsilon'})\log t\} \text{ such that }\\ v_{s_j} \leftrightarrow v_{s_i} \text{ in }G_t(p) \text{ for all } i,j \in [k]
\end{array}
\right \rbrace.
\end{equation}
In the context of the doubly labeled tree process, the event $A_{p, \varepsilon', m, k}$ says that we may find a subset of $k$ vertices of $\cal{T}_t$, all added between times $t^{\varepsilon'}$ and $c_1m(t^{\alpha} + t^{\varepsilon'})\log t$, with the property that the $j$-th vertex $v_{t_j}$ of this subset has its~$U_{t_j}$-label smaller than~$p$. Moreover, when we apply a constant function equals $p$ on $\cal{T}_t$ in the event~$A_{p, \varepsilon', m, k}$, all these~$k$ vertices form a complete graph in the resulting graph $G_t(p)$. 

Theorem~$1$ of~\cite{ARS17} states that setting $k = t^{\alpha}/m$, the event~$A_{p, \varepsilon', m, k}$ occurs with probability at least $1-t^{-\eta}$, for some positive small $\eta$ depending on $p$, which, in our case, is a function of~$\delta$. Thus, for $t$ large enough, we may simply use that
\begin{equation}\label{ineq:probA}
\Pd\left(A_{p, \varepsilon', m, t^{\alpha}/m}\right) \ge 1 - \frac{1}{\log t},
\end{equation}
this bound is good enough to prove Corollary~\ref{cor:cliquesRV}. This is why we are exchanging a polynomial decay by a $\log$ one, to get rid off the dependency on $\delta$ and consequently simplify latter arguments.

Now, using identify \eqref{eq:utj}, the definition of good set and the fact that $f$ is decreasing, we have that
\begin{equation}
    \begin{split}
        \Ed\left[ N\; \middle | \; \mathcal{K} = S, S \text{ is good }\right] & =  p^{-1}\sum_{v_{t_j}  \in S}f(t_j) \ge p^{-1}\sum_{j=t^{\varepsilon'}}^{t^{\alpha}/m}f(c_1m j \log t) \\
        & \ge \frac{F(c_1t^{\alpha}\log t) - F(c_1mt^{\varepsilon'}\log t )}{pc_1m\log t}\\
        & \ge \frac{c_2F(c_1t^{\alpha}\log t)}{\log t}.
    \end{split}
\end{equation}
For some positive constant $c_2$ depending on $\varepsilon',\delta$ and $f$. Moreover, using that the $U$-labels are independent and Chernoff bounds
\begin{equation}\label{eq:chernoff}
    \Pd\left( N \le \frac{c_2F(c_1t^{\alpha}\log t)}{2\log t} \, \middle | \, \mathcal{K} = S, S \text{ is good } \right) \le \exp \Lbrace - \frac{c_2F(c_1t^{\alpha}\log t)}{4\log t}\Rbrace.
\end{equation}
And in order to simplify our writing, let $B$ be the following event
\begin{equation}
    B := \Lbrace \nexists S \subset G_t(p), S \text{ is bad }\Rbrace.
\end{equation}
In the event $B$, all vertices added by the process $\{G_t(p)\}_{t \ge 0}$ did not take too long to be added. That is, the $j$-th vertex added by the process $\{G_t(p)\}_{t \ge 0}$ was added before time $c_1mj\log t$. Also notice that by equation \eqref{eq:bounddeltatauk}, $\Pd(B^c) \le t^{-9}$.

Furthermore, we have the following identity of events
\begin{equation}
    A_{p, \varepsilon', m, t^{\alpha}/m}\cap B = \Lbrace \mathcal{K} \neq \emptyset, \mathcal{K} \text{ is good }\Rbrace
\end{equation}
and for each good $S \subset G_t(p)$ we use $\mathcal{K}(S)$ as a shorthand for
$$
\mathcal{K}(S) := \Lbrace \mathcal{K} = S, S \text{ is good } \Rbrace.
$$
Then, using Equation \eqref{eq:chernoff} and summing over all possible good subsets $S$, we have 
\begin{equation*}
    \begin{split}
        \Pd\left( N \le \frac{c_2F(c_1t^{\alpha}\log t)}{2\log t}, A_{p, \varepsilon', m, t^{\alpha}/m}, B \right) & = \sum_{S}\Pd\left( N \le \frac{c_2F(c_1t^{\alpha}\log t)}{2\log t} \, \middle | \, \mathcal{K}(S) \right) \Pd \left( \mathcal{K}(S)\right) \\
        & \le \exp \Lbrace - \frac{c_2F(c_1t^{\alpha}\log t)}{4\log t}\Rbrace \Pd \left( \mathcal{K} \neq \emptyset, \mathcal{K} \text{ is good }\right).
    \end{split}
\end{equation*}
We then obtain the following bound
\begin{equation}
    \begin{split}
        \Pd\left( N \le \frac{c_2F(c_1t^{\alpha}\log t)}{2\log t} \right) & \le \Pd\left( N \le \frac{c_2F(c_1t^{\alpha}\log t)}{2\log t}, A_{p, \varepsilon', m, t^{\alpha}/m}, B \right) + \Pd\left(A^c_{p, \varepsilon', m, t^{\alpha}/m} \cup B^c\right) \\
        & \le \exp \Lbrace - \frac{c_2F(c_1t^{\alpha}\log t)}{4\log t}\Rbrace  + \frac{2}{\log t},
    \end{split}
\end{equation}
which proves the theorem.
\end{proof}

\section{Proof of Theorems \ref{thm:clique} and \ref{thm:clustering}: clique number and clustering}
In this section we will prove our other two main theorems, which give sharp bounds on the global clustering coefficient and the clique number.
Recall that the global clustering coefficient of $G_t(f)$ is defined as
\begin{equation}\label{def:gcc}
    c^{\rm glob}_t = c^{\rm glob}(G_t(f)) = 3\times \frac{\Delta(G_t(f))}{\Lambda(G_t(f))}.
\end{equation}
For the proof of both theorems will need the bounds on $\Delta(G_t(f))$ and $\Lambda(G_t(f))$ given by the results below. 
\begin{theorem}[Number of Triangles]\label{thm:triangles}Assume that $f$ decreases to zero and it is a regularly varying function with index of regular variation $-\gamma$,  for $\gamma \in [0,1)$. Then, for any $\delta>0$ there exist positive constants $C,C'$ and $C''$ depending on $f$ and $\delta$, such that for all $t$
$$
\Pd \left( C't^{\frac{3(1-\gamma)(1-\delta)}{2}}\le \Delta(G_t(f)) \le Ct^{\frac{3(1-\gamma)(1+\delta)}{2}} \right) \ge 1 - \frac{C''}{\log(t)}
.
$$
\end{theorem}
\begin{theorem}[Number of Cherries]\label{thm:cherries}Assume that $f$ decreases to zero and it is a regularly varying function with index of regular variation $-\gamma$,  for $\gamma \in [0,1)$. Then, for any $\delta>0$ there exist positive constants $C,C'$ and $C''$ depending on $f$ and $\delta$, such that for all $t$
$$
\Pd \left( C't^{2(1-\gamma)(1-\delta)}\le \Lambda(G_t(f)) \le Ct^{2(1-\gamma)(1+\delta)} \right) \ge 1 - \frac{C''}{\log(t)}
.
$$
\end{theorem}
We will postpone their proofs to Section \ref{sec:trianglesandcherries}. For now, we will just focus on how to use them to obtain our two main theorems.
\begin{proof}[Proof of Theorem \ref{thm:clique} (Clique Number)] We will start by the lower bound.

\noindent \underline{Lower bound.} Recall that we already have proven Theorem \ref{t:largecliques} which provides a general lower bound for $\omega_t$ in terms of $\Ed[|G_t|]$ when $f$ is decreasing to zero only. Thus, the lower bound part of Theorem \ref{thm:clique} is a particular case of Theorem \ref{t:largecliques} when $f$ is also regularly varying. Thus, this part of the proof consists essentially of analytical arguments to obtain estimates on~$\Ed[|G_t|]$ itself.

With the above discussion in mind and given $\delta \in (0,1)$, set $\delta'=\delta +\varepsilon$ for some auxiliary $\varepsilon$ to be adjusted later and small enough so $\delta'\in (0,1)$. By Theorem \ref{t:largecliques}, we have that for every~$\delta'\in(0,1)$, there exist positive constants $c_1$ and $c_2$ depending on $\delta'$ and $f$ only such that
\begin{equation}\label{eq:thm1}
\Pd\left( \omega_t \ge \frac{c_2F(t^{(1-\delta')/2})}{\log t}\right) \ge  1-\frac{2}{\log t} - \exp \Lbrace - \frac{c_2F(c_1t^{(1-\delta')/2}\log t)}{4\log t}\Rbrace.    
\end{equation}
Since $f$ is a decreasing to zero, we also have that
\begin{equation}\label{ineq:LBF}
    \Ed[|G_t|] = F(t) = 1 +\sum_{s=2}^tf(s) \ge \int_2^t f(x) {\rm d}x.
\end{equation}
On the other hand, using that $f$ is a regularly varying function at infinity with index $-\gamma$, where $\gamma \in [0,1)$, Theorem \ref{thm:repthm} gives us that $f(t)$ can be written as $f(t)=\frac{\ell(t)}{t^{\gamma}}$, where $\ell$ is a slowly varying function. Thus, by Proposition \ref{prop:factorout}, we have that 
\begin{equation*}
    \lim_{t\to\infty}\frac{\int_2^t f(x) {\rm d}x}{t^{1-\gamma}\ell(t)} = \frac{1}{1-\gamma},
\end{equation*}
which by its turn implies the existence of a  positive and $f$-dependent constant $C$ such that 
\begin{equation}
    \int_2^t f(x) {\rm d}x \ge Ct^{1-\gamma}\ell(t),
\end{equation}
for all $t$. Plugging the above inequality into \eqref{ineq:LBF} yields
\begin{equation}
    \frac{c_2F(t^{(1-\delta')/2})}{\log t} \ge \frac{C\ell(t^{(1-\delta')/2})t^{(1-\delta')(1-\gamma)/2}}{\log t}.
\end{equation}
Notice then that, for any $\varepsilon>0$, Proposition \ref{prop:svsmall}  allows us to replace the slowly varying functions in the RHS of the above inequality by a polynomial factor of $t$, at the expenses of a possibly different $C$, to obtain a nicer lower bound as the one below
$$
\frac{C\ell(t^{(1-\delta')/2})t^{(1-\delta')(1-\gamma')/2}}{\log t} \ge Ct^{(1-\delta' - \varepsilon)(1-\gamma)/2} = Ct^{(1-\delta)(1-\gamma)/2}
$$
Returning to \eqref{eq:thm1}, we have that there exists a positive constant $c_3$ depending on $f,\delta$ and $\varepsilon$ such that 
\begin{equation*}
    \Pd\left( \omega_t \ge Ct^{(1-\delta)(1-\gamma)/2} \right) \ge \Pd\left( \omega_t \ge \frac{c_2F(t^{(1-\delta')/2})}{\log t}\right) \ge 1 -\frac{c_3}{\log t},
\end{equation*}
which is enough to conclude this part of the proof.

The proof will then be completed when we finish the upper bound.

\noindent \underline{Upper bound.}  Recall that by Theorem \ref{thm:triangles}, it follows that for any $\delta>0$ there exist positive constants $C$ and $C''$ depending on $f$ and $\delta$, such that for all $t$
\begin{equation}\label{ineq:thmtriang}
\Pd \left(  \Delta(G_t(f)) \ge Ct^{\frac{3(1-\gamma)(1+\delta)}{2}} \right) \le \frac{C''}{\log(t)}.    
\end{equation}
But notice that if $\omega_t \ge Ct^{\frac{(1-\gamma)(1+\delta)}{2}} $, then $\Delta(G_t(f)) \ge Ct^{\frac{3(1-\gamma)(1+\delta)}{2}}$ for some possibly different constant $C$. This implies the following inclusion of events
$$
\Lbrace \omega_t \ge Ct^{\frac{(1-\gamma)(1+\delta)}{2}} \Rbrace \subset \Lbrace  \Delta(G_t(f)) \ge Ct^{\frac{3(1-\gamma)(1+\delta)}{2}} \Rbrace,
$$
which combined to \eqref{ineq:thmtriang} gives us the desired upper bound w.h.p. and allows us to conclude the proof of Theorem \ref{thm:clique}.
\end{proof}

Now we move to the proof of our second main result of this section.
\begin{proof}[Proof of Theorem \ref{thm:clustering}(Clustering Coefficient)] Notice that the proof is a straightforward consequence of Theorems \ref{thm:triangles} and \ref{thm:cherries} together with the definition of $c^{\rm glob}_t$. 

By Theorem \ref{thm:triangles}, for any fixed $\delta \in (0,1)$, we have that with probability at least $1-C'/\log t$, $\Delta(G_t(f)) \le Ct^{3(1-\gamma)(1+\delta)/2}$. Whereas, by Theorem \ref{thm:cherries}, $\omega_t \ge Ct^{2(1-\gamma)(1-\delta)}$, with probability at least $1-C'/\log t$. These two results combined, gives us that 
$$
c^{\rm glob}_t \le C\times\frac{t^{3(1-\gamma)(1+\delta)/2}}{t^{2(1-\gamma)(1-\delta)}},
$$
with probability at least $1-C'/\log t$, for some constant $C'$ depending on $\delta$ and $f$. This gives the desired upper bound. The lower bound is obtained similarly.
\end{proof}

\subsection{Convergence Theorems and the Inverse Relation}
With the concentration inequalities for the clique number and global clustering given by Theorems \ref{thm:clique} and Theorem \ref{thm:clustering}, respectively, we will show almost surely convergence in the log scale of these statistics properly normalized.
\begin{proof}[Proof of Corollary \ref{cor:cliquesRV}: Convergence of the clique number] We would like to show that 
\begin{equation}
    \lim_{t\to \infty} \frac{\log(\omega_t)}{\log t} = \frac{1 - \gamma}{2}, \, \text{a.s.},
\end{equation}
    when $f$ decreases to zero and is a regularly varying function at infinity with index $-\gamma$, with~$\gamma \in [0,1)$. 

    The proof  will involve passing to a subsequence of $\{\log(\omega_t)/\log t\}_{t\in\mathbb{N}}$ and then using monotonicity of $\omega_t$ and $\log t$. In order to do so, notice that by Theorem \ref{thm:clique}, for every $\delta \in (0, 1)$, there exist positive constants $c_1,c_2$ and $c_3$ depending on $\delta$ only such that
    $$
    \Pd\left( c_1t^{(1-\gamma)(1-\delta)/2} \le \omega_t \le c_2t^{(1-\gamma)(1+\delta)/2}\right) \ge 1 -\frac{c_3}{\log t}.
    $$
    The above result guarantees that for fixed $\delta$ and any $k \in \mathbb{N}$, it follows that 
    \begin{equation}\label{eq:l}
            \Pd \left( \frac{\log(\omega(G_{2^{k^2}}(f)))}{k^2\log 2} \le \frac{(1-\gamma)(1-\delta)}{2} + \frac{\log c_1}{k^2\log 2}  \right) \le \frac{c_3}{k^2\log 2}.
    \end{equation}
    Similarly, we also have that 
    \begin{equation}\label{eq:u}
            \Pd \left( \frac{\log(\omega(G_{2^{k^2}}(f)))}{k^2\log 2} \ge \frac{(1-\gamma)(1+\delta)}{2} + \frac{\log c_2}{k^2\log 2}  \right) \le \frac{c_3}{k^2\log 2}.
    \end{equation}
    Passing to the subsequence $\{\log(\omega(G_{2^{k^2}}(f)))/k^2\log 2\}_{k\in\mathbb{N}}$, we can combine \eqref{eq:l} and \eqref{eq:u} to Borel-Cantelli lemma to obtain that 
    \begin{equation}\label{eq:conv}
        \lim_{k \to \infty} \frac{\log(\omega(G_{2^{k^2}}(f)))}{k^2\log 2} = \frac{1-\gamma}{2}, \text{a.s.},
    \end{equation}
since \eqref{eq:l} and \eqref{eq:u} hold for any $\delta \in (0,1)$  with possibly different constants $c_1,c_2$ and $c_3$.

Finally, using that both $\log t$ and $\omega_t$ are increasing on $t$, we have that for any $t \in [2^{k^2},2^{(k+1)^2}]$, the following inequalities hold
\begin{equation}
    \frac{\omega(G_{2^{k^2}}(f))}{(k+1)^2\log 2} \le \frac{\omega_t}{\log t} \le \frac{\omega(G_{2^{(k+1)^2}}(f))}{k^2\log 2}.
\end{equation}
With the above inequalities, \eqref{eq:conv} and the Squeeze Theorem, one can show that the whole sequence $\{\log(\omega_t)/\log t\}_{t\in\mathbb{N}}$ converges to $\frac{1-\gamma}{2}$ almost surely.
\end{proof}
We now prove similar result for the global clustering coefficient. 
\begin{proof}[Proof of Corollary \ref{cor:clusteringRV}: Convergence of the Clustering Coefficient] The proof of Corollary \ref{cor:clusteringRV} is almost in line with the proof of Corollary \ref{cor:cliquesRV}, for this reason we will just indicate how it follows from Corollary \ref{cor:cliquesRV} proof.

From the definition of $c^{\rm glob}_t$ and properties of the $\log $, it folows that
\begin{equation}\label{eq:logclus}
    \frac{\log c^{\rm glob}_t}{\log t} = \frac{\log 3}{\log t} + \frac{\log \Delta(G_t(f))}{\log t} - \frac{\log \Lambda(G_t(f))}{\log t}.
\end{equation}
    Observe that both statistics $\Delta(G_t(f))$ and $\Lambda(G_t(f))$ are increasing in time. Moreover, Theorems \ref{thm:triangles} and \ref{thm:cherries} provides bounds for the probability of the same order as Theorem \ref{thm:clique}. Thus, the same argument of passing to a subsequence and using Borel-Cantelli holds, which allow one to obtain that
    $$
       \lim_{t\to \infty} \frac{\log \Delta(G_t(f))}{\log t} = \frac{3(1-\gamma)}{2}, \, \text{a.s.,}
    $$
    and that 
    $$
        \lim_{t\to \infty} \frac{\log \Lambda(G_t(f))}{\log t} = 2(1-\gamma), \, \text{a.s.}
    $$
    Plugging the two above equality back into \eqref{eq:logclus} is enough to conclude the proof.
\end{proof}
Finally, the proof of the inverse relation between clustering and clique number.

\begin{proof}[Proof of Corollary \ref{cor:inverserel}] Up to this point, with the results we have proven, the proof of Corollary \ref{cor:inverserel} follows directly from Theorems \ref{thm:clique} and \ref{thm:clustering}. By Theorem \ref{thm:clique}, with probability at least $1- c/\log(t)$ we have that 
$$
    c_1t^{(1-\gamma)(1-\delta)/2} \le \omega_t \le c_2t^{(1-\gamma)(1+\delta)/2}
$$
and by Theorem \ref{thm:clustering}, with probability at least $1-c/\log t$,
$$
c_1t^{-(1-\gamma)(1+\delta)/2} \le c^{\rm glob}_t \le c_2t^{-(1-\gamma)(1-\delta)/2}.
$$
Adjusting the $\delta$'s properly and combining the above results finishes the proof.
\end{proof}

\section{Proof of Theorems \ref{thm:triangles} and \ref{thm:cherries}: Number of Triangles and Cherries}\label{sec:trianglesandcherries}
In this section we will provide bounds on the number of triangles and cherries both counted without multiplicities. 

For the upper bound for both $\Delta(G)$ and $\Lambda(G)$ we will need the result below regarding the number of edges in $\widehat{G}_t$.
\begin{theorem}[Number of edges of $\widehat{G}$]\label{thm:edgesghat}If $f$ decreases to zero and is a regularly varying function at infinity with index $-\gamma$, with~$\gamma\in[0,1)$, then, for every $\delta > 0$ there exists a positive constant $C$ depending on $\gamma$ and $\delta$ only such that
$$
\Ed \left[\mathcal{E}(\widehat{G}_t)\right] \le C t^{1-\gamma + \delta}.
$$
\end{theorem}
For the sake of readability, we will defer the proof of the above result to Section \ref{s:proofedgesghat}. In the next two subsections, we will only focus on how to use Theorem \ref{thm:edgesghat} to obtain bounds for the number of triangles and cherries.

\subsection{Number of Triangles}\label{ss:triangles}In this section we provide bounds on the number of triangles. More specifically, we will prove the following

\begin{proof}[Proof of Theorem \ref{thm:triangles}] We start by the lower bound.

\noindent \underline{\it Lower bound.} It will follow from Theorem \ref{thm:clique}. Recall that if $f$ is a decreasing regularly varying function with index $-\gamma$, we have that for large enough $t$
$$
\frac{F(t)}{\log(t)} \ge Ct^{(1-\gamma)(1-\delta)}
$$
for some constant $C$ depending on $\gamma$, which implies that 
$$
\frac{F(t^{(1-\delta)/2})}{\log(t)} \ge Ct^{(1-\gamma)(1-\delta)/2}.
$$
Thus, by Theorem \ref{thm:clique}, with probability at least $1-3/\log(t)$, there exists a clique in $G_t(f)$ containing at least $Ct^{(1-\gamma)(1-\delta)/2}$ vertices. Consequently, with probability at least $1-3/\log(t)$ we have that
$$
\Delta(G_t(f)) \ge Ct^{\frac{3(1-\gamma)(1-\delta)}{2}},
$$
which gives us the desired lower bound w.h.p.\\
\noindent \underline{\it Upper Bound.} By Markov inequality, we have that 
\begin{equation}\label{ineq:markov}
    \Pd\left(\mathcal{E}(\widehat{G}_t) \ge Ct^{1-\gamma +2\delta} \right) \le \frac{1}{t^\delta}.
\end{equation}
On the event inside the probability above, $\widehat{G}_t$ is a graph on at most $t$ vertices and less than $Ct^{1-\gamma +2\delta}$ edges. 

Now observe that, by Theorem 4 of \cite{R02}, the number of triangles in $\widehat{G}_t$ is deterministically bounded by the number of edges to the power of $3/2$. Combining this to  \eqref{ineq:markov} gives us that 
\begin{equation*}
    \Pd\left(
            \Delta(\widehat{G}_t) \ge \left( Ct^{1-\gamma + 2\delta} \right)^{3/2}
        \right) 
            \le 
                \frac{1}{t^\delta},
\end{equation*}
which is enough to conclude the proof, as $\delta$ is arbitrary.
\end{proof}

\subsection{Number of Cherries}\label{ss:cherries}
In this section we will provide bounds for $\Lambda_t$. Before we start the proof of Theorem \ref{thm:cherries}, we will need some auxiliary results and notation. 

{\bf We now let $i,j,k \in \N$ denote respectively the $i$-th, $j$-th, and $k$-th vertices that are added to $G_t(f)$.}

The lower bound will follow from a lower bound on some vertex's degree, which is proved in \cite{AR22} but stated below for the reader's convenience. 
\begin{lemma}[Theorem 1 of \cite{AR22}]
    \label{thm:deglowerbound}
    Assume that $f$ goes to zero as $t$ goes to infinity. Then, for every $N \in \N$, there exists $C_f >0$ depending on $f$ only such that
    \begin{equation}
    \label{eq:deglowerbound}
        \begin{split}
            \P\left(\forall s\in\N,\exists i \in \{1,2,\cdots,N\}\text{ such that } D_s(i)\geq \frac{\phi(s)}{\phi(N)}\right) &\geq 1- \exp\{-C_f N\},
        \end{split}
    \end{equation}    
    where 
    \begin{equation}
        \label{def:phi}
        \phi(t) = \phi(t, f) := \prod_{s = 1}^{t - 1} \Big( 1 + \frac{1}{s} - \frac{f(s + 1)}{2s} \Big).
    \end{equation}
\end{lemma}
We will need the following technical Lemma on the function $\xi := \phi(t)/t$:
\begin{lemma}[Lemma 2.2 of \cite{AR22}]
    \label{l:xislow}
    Let
    \begin{equation}\label{def:xi}
        \xi(t) := \frac{\phi(t)}{t} 
                = \prod_{s = 1}^{t - 1} \frac{s}{s + 1} \Bigg( 1 + \frac{1}{s} - \frac{f(s + 1)}{2s} \Bigg)
                = \prod_{s = 1}^{t - 1} \Bigg( 1 - \frac{f(s + 1)}{2(s + 1)} \Bigg)
    \end{equation} 
    Then, if $f$ decreases to $0$, $\xi$ is a decreasing slowly varying function uniformly bounded from above by $1$.
\end{lemma}

The above lemma together with Proposition \ref{prop:svsmall} implies that for any $\delta>0$, there exist a positive constant $C$ depending on $\delta$ and $f$, such that for all $t\in \mathbb{N}$
\begin{equation}\label{ineq:phibound}
    Ct^{1-\delta} \le \phi(t) \le t.
\end{equation}
With the above discussion in mind, we can now move to the proof of Theorem \ref{thm:cherries}.
\begin{proof}[Proof of Theorem \ref{thm:cherries}] We will obtain the upper bound first.

\noindent \underline{\it Upper Bound.} For the upper bound, we use Theorem \ref{thm:edgesghat} again. The crucial observation to be made is that, on the event $\{\mathcal{E}(\widehat{G}_t) \le C\log(t)t^{1-\gamma + \delta}\}$, the graph observable $\Lambda(\widehat{G}_t)$ is bounded from above by $\left(C\log^5(t)t^{1-\gamma}\right)^2$. Thus, by \eqref{ineq:markov}, it follows that 
$$
\Pd\left(\Lambda(\widehat{G}_t) > \left(C\log(t)t^{1-\gamma + 
\delta}\right)^2\right) \le \frac{1}{\log(t)},
$$
which is enough to conclude the proof of the upper bound.

\noindent \underline{\it Lower Bound.} The main idea behind the proof of the lower bound is to  show that for any small enough $\delta$, with high probability, $G_t$ has a vertex with at least $Ct^{1-\gamma-\delta}$ neighbors for some constant $C$. Then, $\binom{Ct^{1-\gamma-\delta}}{2}$ is a lower bound for $\Lambda_t$.

In order to do so, for a fixed $t$ and small enough $\delta$,  make $N= 10C_f^{-1}\log t$ on Lemma \ref{thm:deglowerbound} to obtain
\begin{equation}\label{ineq:probdeg}
       \P\left(\exists i \in \{1,2,\cdots,N\}\text{ such that } D_t(i)\geq \frac{\phi(t)}{\phi(N)}\right) \geq 1- t^{-10}.
\end{equation}
Now let $v_*$ be the vertex of smallest index whose degree at time $t$ is larger than $\phi(t)/\phi(N)$. Notice that by \eqref{ineq:phibound} and our choice of $N$ we have that
\begin{equation}\label{ineq:boundphiN}
    \frac{\phi(t)}{\phi(N)} \ge \frac{Ct^{1-\delta/2}}{N} \ge Ct^{1-\delta}.
\end{equation}
For any time $s \in \{t,\dots,2t\}$ denote by $A_s^*$ the following event
\begin{equation}
    A_s^* := \left \lbrace Z_s = 1, v_s \text{ connects to }v_* \right \rbrace.
\end{equation}
In words, $A_s^*$ is the event in which at time $s$ a new vertex is added and connected to the $v_*$. Then, let $Y_t$ be the following random variable
\begin{equation}
    Y_t := \sum_{s=t}^{2t} \mathbb{1}_{A_s^*},
\end{equation}
which counts how many new neighbors $v_*$ gained from time $t$ to $2t$. Also, let $\P_*$ be the following conditional measure
\begin{equation}
    \P_*(\cdot) := \Pd \left( \; \cdot \; \middle | \; \exists i \in \{1,2,\cdots,N\}\text{ such that } D_t(i)\geq \frac{\phi(t)}{\phi(N)} \right).
\end{equation}
Now, using that $f$ is decreasing, that vertices' degree are increasing in time and that $Z_{s}$ is independent of the whole past, we obtain
\begin{equation}\label{ineq:prob}
    \begin{split}
        \P_* \left( A_{s}^* \middle | \mathcal{F}_{s-1}\right) = f(s)\frac{D_{s-1}(v_*)}{2s} \ge f(2t)\frac{Ct^{1-\delta}}{4t} \ge Cf(2t)t^{-\delta}.
    \end{split}
\end{equation}
By the Representation Theorem (Theorem \ref{thm:repthm}), $f(t) =\ell(t)t^{-\gamma}$, where $\ell$ is a slowly varying function. And by Proposition \ref{prop:svsmall} we can then absorb the slowly varying factor of $f$ into a polynomial factor in the last inequality of \eqref{ineq:prob} by increasing the value of $\delta$ and changing the constant $C$. This allows us to obtain
\begin{equation}\label{ineq:probf}
    \P_* \left( A_{s}^* \middle | \mathcal{F}_{s-1}\right) \ge \frac{C}{t^{\gamma + \delta}}.
\end{equation}
Thus, under $\Pd_*$, $Y_t$ dominates a random variable $W$ distributed as ${\rm bin}(t,Ct^{-\gamma -\delta})$. This fact combined with Chernoff bounds for $W$ yields
\begin{equation}\label{ineq:yt}
    \Pd_* \left( Y_t \le \frac{Ct^{1-\gamma -\delta}}{2} \right) \le \Pd_* \left( W \le \frac{Ct^{1-\gamma -\delta}}{2} \right) \le \exp \{-Ct^{1-\gamma -\delta}/8\}.
\end{equation}
The above inequality tells us that conditioned on the event that at time $t$ there exists a vertex whose degree is at least $t^{1-\delta}$, then, with probability at least $1-\exp \{-Ct^{1-\gamma -\delta}/8\}$, at time $2t$ this vertex has at least~$Ct^{1-\gamma-\delta}$ neighbors. Then, combining \eqref{ineq:yt} with \eqref{ineq:probdeg} yields
\begin{equation}
    \Pd\left( \nexists v_* \in G_{2t}, v_* \text{ has at least }Ct^{1-\gamma-\delta}\right) \le \frac{2}{t^{10}},
\end{equation}
which allows us to conclude that with probability at least $1-2t^{-10}$,
$$
\Lambda_{2t} \ge \binom{Ct^{1-\gamma-\delta}}{2}
$$
and consequently to finish the proof of the theorem.

\end{proof}

\section{Proof of Theorem \ref{thm:edgesghat}: Number of edges in $\widehat{G}_t$}\label{s:proofedgesghat}

This section is devoted to prove Theorem \ref{thm:edgesghat}, which is the last result left to be proven. As usual, we will need a couple of technical estimates involving the time of appearance of vertices, a tight upper bound for their degrees together with estimates on the average number of edges between two vertices. These are all technical estimates and their proof might be quite technical. In a first reading, the reader might find instructive start from Section \ref{ss:proofedgeshat} which contains the proof of Theorem \ref{thm:edgesghat}.

It is also important to highlight that throughout this entire section, $f$ will be  a regularly varying function with index of variation at infinity $\gamma \in [0, 1)$, which decreases to $0$ monotonically.

Throughout this section we will make need new notation. Together with the functions $\phi$ and $\xi$ defined in (\ref{def:phi}, \ref{def:xi}), we will need the function~$F^{-1}:\R\to\R$ defined by
\begin{equation}
    \label{eq:Finvdef}
    F(r):=\sum_{1\leq s \leq r} f(s);\quad\quad F^{-1}(r):=\inf_{s \in \R}\{     F(s) \geq r \}.
\end{equation}
Since we are assuming~$f(s)$ to be strictly positive for all~$s$, and since~$\|f\|_\infty \leq 1$, we have that
\begin{equation}
\label{eq:Finvdef2}
r\leq F(F^{-1}(r))\leq r+1,\quad\text{ and } \quad r\leq F^{-1}(F(r))\leq r+1.
\end{equation}
We define $\tau_i$ to be the first time such that the cardinality of the vertex set equals $i$, that is
\begin{equation}
    \tau_i = \inf_{t \in \N} \big\{ t; \, |G_{t}(f)| = i \big\},
\end{equation}
and in order to make notation less cumbersome, we further define
\begin{equation}\label{def:psi}
    \psi(i) = \phi\left( F^{-1}(i)\right).
\end{equation}

\subsection{Time of Appearance of Vertices}

\begin{lemma}[Time of Appearance of Vertices]\label{lemma:tauibound}Assume that $f$ decreases to zero and it is a regularly varying function with index of regular variation $-\gamma$,  for $\gamma \in [0,1)$. Then, there exists a constant~$C_f>0$ such that for~$i\in\N$ and every~$\delta>0$
\begin{equation}
\label{eq:tauiconc}
\P\left((1-\delta)F^{-1}(i)\le \tau_i    \leq  (1+\delta)F^{-1}(i)     \right) \ge 1- C_f^{-1}\exp\left\{   -C_f  (1-\gamma)\delta  i      \right\}.
\end{equation}
\end{lemma}
\begin{proof}
We will prove that there exist a constant $C'_f$ such that 
\begin{equation}\label{eq:taulower}
	\P\left( \tau(i)     \leq  (1-\delta)F^{-1}(i)     \right) \le (C'_f)^{-1}\exp\left\{   -C'_f  (1-\gamma)\delta  i      \right\}.
\end{equation}
The upper bound for the probability of the event $\{\tau(i) \ge (1+\delta)F^{-1}(i)\}$ will follow from the exactly same argument.

By Theorem \ref{thm:repthm}, we know that~$f(t)=\ell(t)t^{-\gamma}$ with~$\ell$ being a slowly varying function. Proposition \ref{prop:factorout} implies that as~$r\to\infty$, using the monotonicity of~$f$ to approximate the integral by the sum,
\begin{equation}
\label{eq:cotaF2}
\begin{split}
F((1-\delta)r)
 &=\sum_{s=1}^{\lfloor  (1-\delta) r  \rfloor}      f(s)   \\  
  & =  \sum_{s=1}^{\lfloor (1-\delta) r   \rfloor}      \ell(s)s^{-\gamma}   \sim \frac{1}{1-\gamma}\ell((1-\delta) r)(1-\delta)^{1-\gamma}r^{1-\gamma} \\
  & \sim (1-\delta)^{1-\gamma} F(r),
\end{split}
\end{equation}
where we also used the fact that~$\ell$ is slowly varying to obtain~$\ell((1-\delta)r)\sim\ell(r)$. Letting $i_* = (1-\delta)F^{-1}(i)$
and recalling that~$|G_r(f)|$ is a sum of independent random variables, we may use the above equation, \eqref{eq:Finvdef2}, and an elementary Chernoff bound to obtain 
\begin{equation}
\label{eq:tauiconc2}
\begin{split}
\lefteqn{\Pd\left(\tau(i)     \leq  (1-\delta)F^{-1}(i)     \right)}\quad\quad
\\
&=
\Pd\left( |G_{i_*}(f)|  \geq i     \right)
\\
&=
\Pd\left( |G_{i_*}(f)|  \geq \left(1+\left(i \left( F\left((1-\delta)F^{-1}(i)\right)\right)^{-1}-1\right)\right)\Ed\left[  |G_{i_*}(f)|       \right]    \right)
\\
&\leq 
\exp\left\{   -\frac{1}{3}   \left(i \left( F\left((1-\delta)F^{-1}(i)\right)\right)^{-1}-1\right)  F\left( (1-\delta)F^{-1}(i)\right)          \right\}
\\
&\leq 
\exp\left\{   -\frac{1}{3}  \left(   (1-\delta)^{-(1-\gamma)}      -1\right)(1-\delta)^{1-\gamma} \cdot i      \right\}
\\
&\leq 
\exp\left\{   -C(1-\gamma)\delta i      \right\},
\end{split}
\end{equation}
for sufficiently large~$i$. By adjusting the constant we can then show~\eqref{eq:taulower} for every~$i\geq 1$, finishing the proof.

\end{proof}

We will also need a suitable upper bound for the degree of the $i$-th vertex:

\begin{lemma}[Upper bound on vertices degrees]\label{lemma:upperbounddeg2}
	Assume that $f$ decreases to zero and it is a regularly varying function with index of regular variation $-\gamma$,  for $\gamma \in [0,1)$. Then, there exists a constant~$C_f>0$ such that, for every~$\alpha>1$,
	\begin{equation*}
	\label{eq:upperbounddeg2}
	\begin{split}
	\P\left(       \exists s\in\N \text{ such that } D_s(i) \geq   \alpha\frac{\phi(s)}{\phitil(i)}                     \right) & \leq C_f^{-1}\left(\exp\left\{   -C_f \frac{\alpha-1}{\alpha} i    \right\}     +  \exp\left\{   -C_f \cdot  \alpha    \right\} \right)    . 
	\end{split}
	\end{equation*}
\end{lemma}
\begin{proof}
Lemma $5$ of \cite{alvesribeirovalesin23} implies
\begin{equation}
    \label{e:valesin_deg_upperbound}
    \begin{split}
        \P\left(       \exists s\in\N \text{ such that } D_s(i) \geq   \alpha\frac{\phi(s)}{\phi(n)} \Bigg| \tau_i = n                    \right) 
            & \leq 
                C_f^{-1}  \exp\left\{   -C_f \cdot  \alpha    \right\} .
    \end{split}
\end{equation}
We note that, for~$r>0$ and~$\delta\in(0,1)$,
\begin{equation}
\label{eq:phiFphiFsobre2}
\frac{\phi\left(r\right)}{\phi\left((1-\delta)r\right)}=\prod_{s=(1-\delta)r}^{r-1}
\left(     1+\frac{1}{s}              -\frac{f(s+1)}{2s}       \right) \leq 
\prod_{s=(1-\delta)r}^{r-1}
  \frac{s+1}{s}       \leq \frac{1}{1-\delta}.       
\end{equation}
Let~$\delta = (\alpha-1)/(2\alpha)$, so that~$\alpha(1-\delta)>1$ and~$\delta\in(0,1)$. We obtain, by the above equation, monotonicity, and the union bound,
\begin{equation}
\nonumber
\label{eq:upperbounddeg3}
\begin{split}
\lefteqn{\P\left(       \exists s\in\N \text{ such that } d_s(i) \geq   \alpha\frac{\phi(s)}{\phi\left(F^{-1}(i)\right)}              \right) \leq  
\P\left(\tau(i)     \leq \left(1 -\delta\right) F^{-1}(i)     \right)}
\\
&\quad +
\P\left(       \exists s\in\N \text{ such that } d_s(i) \geq   \alpha(1-\delta)\frac{\phi(s)}{\phi\left((1-\delta)F^{-1}(i)\right)}  \middle |  \tau(i)     = (1-\delta)F^{-1}(i)           \right),
\end{split}
\end{equation}
The result then follows from Equation \ref{e:valesin_deg_upperbound} and Lemma~\ref{lemma:tauibound}.
\end{proof}

\subsection{Average Number of Edges between two vertices}

Our next step is to prove the following lemma, concerning the number of edges between the $i$-th and $j$-th vertex in $G_s$, which we denote $e^{ij}_s$.

\begin{lemma}[Average Number of Edges between two vertices]\label{lemma:expeijt} There exist positive constants $C_1$ and $C_2$ depending on $f$ only, such that, for all $1\le s \le t$ and $j>i>\log^2(t)$, 
    $$
        \Ed\left[e_s^{ij}\right]\le C_1s^2e^{-C_2\log^2(t)} + \frac{C_3\log^2(t)}{2\psi(i)} + \frac{\log^4(t)s}{\phitil(i)\phitil(j)}.
    $$
\end{lemma}
\begin{proof}
    Throughout this proof we will need to define a new filtration which, at time $r$, incorporates not only all the random choices made by the process until that time, but also if time $r + 1$ corresponded to a vertex or edge step:
    \begin{equation}
        \mathcal{F}_r^+ = \sigma(\mathcal{F}_r, Z_{r+1}).
    \end{equation}
    At the precise time in which vertex $j$ is added to the graph, the only way for the the number of edges between $i$ and $j$ to increase is if $j$ connects to $i$ in a vertex step. Therefore,
    \begin{equation}\label{eq:tj}
        \Ed \left[ \Delta e_{t_j-1}^{ij}\mathbb{1} \{\tau_j = t_j \} \middle | \mathcal{F}^+_{t_j-1} \right] =    \frac{D_{t_j-1}(i)}{2(t_j-1)}\mathbb{1} \{\tau_j = t_j \}.
    \end{equation}
    For $r\ge t_j$, however, an edge step must occur in order for this random variable to increase,
    \begin{equation}\label{eq:t}
        \Ed \left[ \Delta e_{r}^{ij}\mathbb{1} \{\tau_j = t_j \} \middle | \mathcal{F}^+_{r} \right] = \frac{(1-Z_{r+1})D_{r}(i)D_r(j)}{2r^2}\mathbb{1} \{\tau_j = t_j \}.
    \end{equation}
    We define the event $B^i_t$, where the degree of vertex $i$ is well behaved:
    \begin{equation}\label{eq:bti}
        B^{i}_t = \Lbrace   \exists s\in\N \text{ such that } D_s(i) \geq   \log^2(t)\frac{\phi(s)}{\phitil(i)}  \Rbrace.
    \end{equation}
    Lemma~\ref{lemma:upperbounddeg2} implies $\Pd(B^i_t) \le C'e^{-C\log^2(t)}$, for some positive constants $C$ and $C'$ depending on~$f$. 

    Using that $D_{t_j -1}(i) \le t_j$ deterministically, we obtain
    \begin{equation}
        \begin{split}
            \Ed \left[ \Delta e_{t_j-1}^{ij}\mathbb{1} \{\tau_j = t_j \} \middle | \mathcal{F}^+_{t_j-1} \right]\mathbb{1}_{B_t^i} \le    \mathbb{1}_{B_t^i} \mathbb{1}\{\tau_j = t_j \},
        \end{split}
    \end{equation}
    which together with \eqref{eq:tj}, implies
    \begin{equation}
        \Ed \left[ \Delta e_{t_j-1}^{ij}\mathbb{1} \{\tau_j = t_j \}\right] \le \frac{\log^2(t)\phi(t_j-1)}{2\phitil(i)t_j}\Pd\left(\tau_j = t_j\right) + C'e^{-C\log^2(t)};
    \end{equation}
    By the same reasoning for \eqref{eq:t} 
    \begin{equation}
        \Ed \left[ \Delta e_{r}^{ij}\mathbb{1} \{\tau_j = t_j \}  \right] \le \frac{\log^4(t)\phi(r)^2}{2\phitil(i)\phitil(j)r^2}\Pd\left(\tau_j = t_j\right) + 2C'e^{-C\log^2(t)}.
    \end{equation}
    Since $\phi(r) \leq r$, it follows that
    \begin{equation}\label{ineq:s1}
        \sum_{r=t_j-1}^s \frac{\log^4(t)\phi(r)^2}{2\phitil(i)\phitil(j)r^2} \le \frac{\log^4(t)s }{\phitil(i)\phitil(j)}.
    \end{equation}
    Moreover, combining Lemma \ref{lemma:tauibound} with $\delta = 1/2$ with the same bound on $\phi$, we have that
    \begin{equation}\label{ineq:c1}
        \begin{split}
            \sum_{t_j=j}^{F^{-1}(j)/2}\frac{\log^2(t)\phi(t_j-1)}{2\phitil(i)t_j}\Pd\left(\tau_j = t_j\right) 
                & = \frac{\log^2(t)}{2\phitil(i)}\sum_{t_j=j}^{F^{-1}(j)/2}\Pd\left(\tau_j = t_j\right)      \\
                & \le (C'_f)^{-1}\frac{\log^2(t)}{2\phitil(i)}e^{-C_f'(1-\gamma)\log^2(t)/2} 
                \\
                &\le C''e^{-C\log^2(t)},
        \end{split}
    \end{equation}
    for two constants $C''$ and $C$ depending on $f$ only.
    The same reasoning gives us similar bound for other two constants $C_2''$ and $C_2$,
    \begin{equation}\label{ineq:c2}
        \begin{split}
            \sum_{t_j=3F^{-1}(j)/2}^{t}\frac{\log^2(t)\phi(t_j-1)}{2\phitil(i)t_j}\Pd\left(\tau_j = t_j\right) & \le  C_2''e^{-C_2\log^2(t)}.
        \end{split}
    \end{equation}
    Using one more time the bound on $\phi$, we have 
    \begin{equation}\label{ineq:c3}
        \begin{split}
            \sum_{t_j=F^{-1}(j)/2}^{3F^{-1}(j)/2}\frac{\log^2(t)\phi(t_j-1)}{2\phitil(i)t_j}\Pd\left(\tau_j = t_j\right) & \le \frac{\log^2(t)}{2\phitil(i)} .
        \end{split}
    \end{equation}
    Finally, combining \eqref{ineq:s1}, \eqref{ineq:c1}, \eqref{ineq:c2} and \eqref{ineq:c3}, we obtain
    \begin{equation}
        \begin{split}
            \sum_{t_j=j}^t\Ed & \left[e_t^{ij}\mathbb{1} \{\tau_j = t_j \}\right]  \\ 
            & = \sum_{t_j=j}^s\sum_{r=t_j-1}^s\Ed\left[\Delta e_r^{ij}\mathbb{1} \{\tau_j = t_j \}\right] \\
        & \le 3s^2C'e^{-C\log^2(t)} + \sum_{t_j=j}^s\left[\frac{\log^2(t)\phi(t_j-1)}{2\phitil(i)t_j} + \frac{\log^4(t)s\xi^2(s)}{\phitil(i)\phitil(j)}\right]\Pd\left(\tau_j = t_j\right) \\
        & \le C_4's^2e^{-C_4\log^2(t)} + \frac{\log^2(t)}{2\phitil(i)} + \frac{\log^4(t)s}{\phitil(i)\phitil(j)},
        \end{split}
    \end{equation}
    for some positive constants $C_4'$ and $C_4$ depending on $f$ only, which concludes the proof.
    
\end{proof}

\subsection{Proof of Theorem \ref{thm:edgesghat}}\label{ss:proofedgeshat}

Finally we can prove Theorem \ref{thm:edgesghat} which is the last result left to be proven.
\begin{proof}[Proof of Theorem \ref{thm:edgesghat}] We begin observing that the result is clearly true when $\gamma = 0$, for this reason, we will assume $\gamma \in (0,1)$. In this case, since $f(s)/s$ is summable, by the definition of $\xi$ in \ref{def:xi}, we we know that that there exists $c_f > 0$ such that
\begin{equation}
    \label{eq:phi_linear_bound}
    c_f t \leq \phi(t) \leq t
\end{equation}
for all $t > 0$. Now, recall that $|G_t|$ can be written as follows
$$
|G_t| = 1 + \sum_{s=1}^t Z_s,
$$
where $Z_s \sim {\rm Ber}(f(s))$. Thus, by Chernoff bounds, there exist some positive constant $C$ such that
\begin{equation}
    \label{eq:num_vertices_chernoff}
    \Pd \left( |G_t| > C t^{1 - \gamma} \right)
        \leq 
            C^{-1} \exp \left\{ C t^{1 - \gamma} \right\}.
\end{equation}

Notice that $\mathcal{E}(\widehat{G}_t)$ satisfies the following upper bound
\begin{equation}\label{ineq:boundeg}
    \mathcal{E}(\widehat{G}_t) = \sum_{\substack{i,j\in G_t, \\ i<j}}\mathbb{1}\{e^{ij}_t \ge 1\} \le \log^4(t) \;  + \sum_{\substack{i,j\in G_t, \\ \log^2(t) < i<j}}\mathbb{1}\{e^{ij}_t \ge 1\},
\end{equation}
almost surely. Since at time $t$ we have added $t$ edges, the second term on the RHS of the above inequality can be bounded as follows
\begin{equation}\label{ineq:bound2}
    \sum_{\substack{i,j\in G_t, \\ \log^2(t) < i<j}}\mathbb{1}\{e^{ij}_t \ge 1\} \le \sum_{\substack{i,j\in G_t, \\ \log^2(t) < i<j}}\mathbb{1}\{e^{ij}_t \ge 1, |G_t| \le Ct^{1-\gamma}\} + t\cdot\mathbb{1}\{|G_t| \ge Ct^{1-\gamma}\}
\end{equation}
Moreover, for a fixed $i\ge \log^2(t)$, we also have
\begin{equation}\label{ineq:boundj}
    \sum_{\substack{j\in G_t, \\ j \ge i}}\mathbb{1}\{e^{ij}_t \ge 1, |G_t| \le Ct^{1-\gamma}\} \le \frac{t^{1-\gamma}}{i} + \sum_{j=t^{1-\gamma}/i}^{Ct^{1-\gamma}}e^{ij}_t.
\end{equation}
Plugging the above inequality back on \eqref{ineq:bound2} back and then on \eqref{ineq:boundeg} yields
\begin{equation*}
    \mathcal{E}(\widehat{G}_t) \le \log^4(t)+ t\cdot\mathbb{1}\{|G_t| \ge Ct^{1-\gamma}\} + \sum_{i= \log^2(t)}^{Ct^{1-\gamma}}\left[\frac{t^{1-\gamma}}{i} + \sum_{j=t^{1-\gamma}/i}^{Ct^{1-\gamma}}e^{ij}_t \right].
\end{equation*}
Thus, taking the expectation on both sides of the above inequality and using \eqref{eq:num_vertices_chernoff} gives us
\begin{equation}\label{ineq:boundexp1}
    \Ed \left[ \mathcal{E}(\widehat{G}_t) \right] \le \log^4(t) +t\cdot e^{-Ct^{1-\gamma}}+C\log(t)t^{1-\gamma} + \sum_{i= \log^2(t)}^{Ct^{1-\gamma}} \sum_{j=t^{1-\gamma}/i}^{Ct^{1-\gamma}}\Ed\left[e^{ij}_t\right].
\end{equation}
To bound the last term of the above inequality, we apply Theorem \ref{lemma:expeijt}. We first need, however, to better understand the behavior of the function $\psi = \phi \circ F^{-1}$. Theorem \ref{thm:inverse} and elementary integration implies that $F^{-1}$ is a regularly varying function with index of regularity $1/(1-\gamma)$. Given $r \in \N$, we then have
\begin{equation}
    \psi(r) = \phi(F^{-1}(r)) \geq c F^{-1}(r) \geq c r^{\frac{1}{1-\gamma} - \delta},
\end{equation}
for any $\delta > 0$ and a positive constant $c$ depending on $\delta$. This yields, together with Equation \eqref{ineq:boundexp1},
\begin{equation}\label{ineq:last}
    \begin{split}
        \sum_{i= \log^2(t)}^{Ct^{1-\gamma}} \sum_{j=t^{1-\gamma}/i}^{Ct^{1-\gamma}}\Ed\left[e^{ij}_t\right] & \le \sum_{i= \log^2(t)}^{Ct^{1-\gamma}} \sum_{j=t^{1-\gamma}/i}^{Ct^{1-\gamma}}C_1t^2e^{-C_2\log^2(t)} + \frac{C_3 t^\delta}{i^{\frac{1}{1-\gamma}}}+\frac{C_4 t^{1 + \delta}}{(ij)^{\frac{1}{1-\gamma}}}\\
        & \le C_1t^3e^{-C_2\log^2(t)} + C_3t^{1-\gamma +\delta}+\sum_{i= \log^2(t)}^{Ct^{1-\gamma}}\frac{C_4 t^{1 + \delta}}{i^{\frac{1}{1-\gamma}}}\cdot \frac{i^{\frac{\gamma}{1-\gamma}}}{t^{\gamma}}\\
        & \le C_1t^3e^{-C_2\log^2(t)} + C_3t^{1-\gamma + \delta}+ C_4 t^{1-\gamma + \delta} \\
        & \le Ct^{1-\gamma + \delta}.
    \end{split}
\end{equation}
By replacing the above bound on \eqref{ineq:boundexp1}, we obtain the desired result.
\end{proof}

\appendix
\section{Important results for Regularly varying functions}
In this appendix we gather important results from the Karamata's theory. All of them may be found in \cite{regvarbook}, but for the reader's convenience we state them here under our notation.

\begin{proposition}[Proposition 1.3.6 of \cite{regvarbook}]\label{prop:svsmall}Let $\ell$ be a slowly varying function. Then, for any $a>0$ it follows that
    $$
    x^a\ell(x) \rightarrow \infty, \quad x^{-a}\ell(x) \rightarrow 0, \text{ as }x \to \infty.
    $$
\end{proposition}

\begin{theorem}[Theorem 1.4.1 of \cite{regvarbook}]\label{thm:repthm}Let $f$ be a strictly positive regularly varying function with index $\beta$. Then, there exists a slowly varying function $\ell$ such that
$$
f(x) = \ell(x)x^{\beta}.
$$
\end{theorem}
\begin{proposition}[Proposition 1.5.8 of \cite{regvarbook}]\label{prop:factorout}Let $\ell$ be a locally bounded slowly varying function. If $\alpha>-1$, then
    $$
    \frac{\int_{x_0}^{x}t^\alpha\ell(t)\mathrm{d}t}{x^{1+\alpha}\ell(x)} \longrightarrow \frac{1}{1+\alpha}\quad \text{as }x\to \infty.
    $$
    
\end{proposition}

\begin{proposition}[Proposition 1.5.9a of \cite{regvarbook}]\label{prop:sumsv}Let $\ell$ be a locally bounded slowly varying function. Then, the function $\widehat{\ell}$ defined as
    $$
    \widehat{\ell}(x):= \int_X^{x}\frac{\ell(t)\mathrm{d}t}{t} 
    $$
    is a slowly varying function. Moreover,
    $$
    \frac{\widehat{\ell}(x)}{\ell(x)} \longrightarrow \infty \quad \text{as }x\to \infty.
    $$
    
\end{proposition}

\begin{proposition}[Proposition 1.5.10 of \cite{regvarbook}]\label{prop:sumfinite}Let $\ell$ be a slowly varying function. If $\alpha<-1$, then
    $$
    \frac{\int_x^{\infty}t^\alpha\ell(t)\mathrm{d}t}{x^{1+\alpha}\ell(x)} \longrightarrow \frac{1}{-1-\alpha}\quad \text{as }x\to \infty.
    $$
    
\end{proposition}
If $f$ a locally bounded function defined at $[x_0, \infty)$ and tends to $\infty$ as $x$ tends to $\infty$, then its generalized inverse 
$$
f^{\leftarrow}(x):= \inf\{y \in [x_0, \infty) : f(y)>x \}
$$
is defined on $[f(x_0), \infty)$ and it increases monotonically to infinity. The theorem below gives additional information on the generalized inverse when $f$ is a regularly varying function.
\begin{theorem}[Theorem 1.5.12 of \cite{regvarbook}]\label{thm:inverse}Let $f$ be a regularly varying function with index of regular varying $\beta > 0$. Then $f^{\leftarrow}$ is a regularly varying function with index $1/\beta$.
\end{theorem}


\bibliography{ref}
\bibliographystyle{plain}

\end{document}